\documentclass[12pt]{article}
\usepackage[utf8]{inputenc}
\usepackage{geometry}
\usepackage[colorlinks]{hyperref}

\usepackage{graphicx}

\usepackage[bb=boondox]{mathalfa}  

\usepackage{amssymb,amsmath,amsthm}
\usepackage{mathtools}
\usepackage{mathrsfs}  
\usepackage{fancyhdr}
\usepackage{enumitem}
\usepackage{bbm}
\usepackage{dsfont}

\newcommand{\N}{\mathbb{N}}
\newcommand{\Z}{\mathbb{Z}}
\newcommand{\R}{\mathbb{R}}

\newcommand{\E}{\mathbb{E}}

\newcommand{\Prob}{\mathbb{P}}

\numberwithin{equation}{section}

\newcommand{\1}{\mathbbm{1}}

\newtheorem{theorem}{Theorem}[section] 
\newtheorem{corollary}[theorem]{Corollary}
\newtheorem{proposition}[theorem]{Proposition}
\newtheorem{lemma}[theorem]{Lemma}

\theoremstyle{definition}

\theoremstyle{remark}
\newtheorem{remark}[theorem]{Remark}

\usepackage{xcolor}
\usepackage[normalem]{ulem}

\newcommand{\Sep}{\textnormal{Sep}}
\newcommand{\dist}{\textnormal{dist}}
\newcommand{\taubreak}{\tau_{\textnormal{break}}}

\title{The broadcast voter model: Stationary measures}
 
\author{Jhon Astoquillca \thanks{Institute of Mathematics, Statistics and Computer Science, University of S\~ao Paulo} \and Adrián González Casanova \thanks{School of Mathematical and Statistical Sciences, Arizona State University} \and Renato S. dos Santos \thanks{Department of Mathematics, Federal University of Minas Gerais}} 
\date{\today}

\begin{document}

\maketitle

\begin{abstract} 
The classical voter model is dual to coalescing random walks, in much the same way that the Kingman coalescent describes genealogies through binary mergers. Motivated by the passage from the Kingman to the~\(\Lambda\)-coalescent, where multiple mergers are allowed, we introduce the~\(\Lambda\)-broadcast voter model, in which a voter may transmit its opinion simultaneously to several neighbors according to a mechanism governed by a finite measure~\(\Lambda\) on~\([0,1]\); the classical voter model is recovered when~\(\Lambda=\delta_0\). We show that our model is well-defined on general graphs with bounded degrees, and we construct a family of stationary measures~\({\mu_{\alpha,\Lambda}}\), parametrized by bounded harmonic functions, that characterize the extremal stationary measures according to the finite or infinite collision property of the underlying random walk. We then study the dependence of the stationary measures on the broadcasting measure $\Lambda$ through a finite number of its moments. We give general conditions under which different moment profiles yield different stationary measures and obtain, in particular, a complete characterization on trees with bounded degrees and on nearest-neighbour~\(\mathbb Z^d\),~\(d\ge3\). We also prove continuity of the stationary measures with respect to the moment profile. Our main tool is a dual system of coordinated coalescing random walks, whose individual particles are simple random walks but which may jump and coalesce simultaneously; a coupling with independent random walks allows us to relate their long-time behaviour to ordinary random walk collision properties.

\end{abstract}



\section{Introduction}

The voter model is one of the most basic interacting particle systems for the spread of
opinions, and one that manages to combine an almost minimal description with a
remarkably rich behaviour. Each vertex of a graph carries one of two opinions, and the
dynamics consists of local copying events: at rate one along each edge, a vertex adopts
the opinion of a neighbor. Simple as this rule is, it has been the source of deep
questions in probability theory, largely because of the model's duality with coalescing
random walks. Tracing an opinion backward in time produces a random walk, and ancestral
lines coalesce whenever they meet. Through this duality, the collision and recurrence
properties of random walks on the underlying graph become the main tool for describing
the long-time behaviour of the process and, in particular, the structure of its
stationary measures. See, for instance, \cite{Liggett85, Swart2026}.

\begin{figure*}
    \centering
\includegraphics[width=15cm]{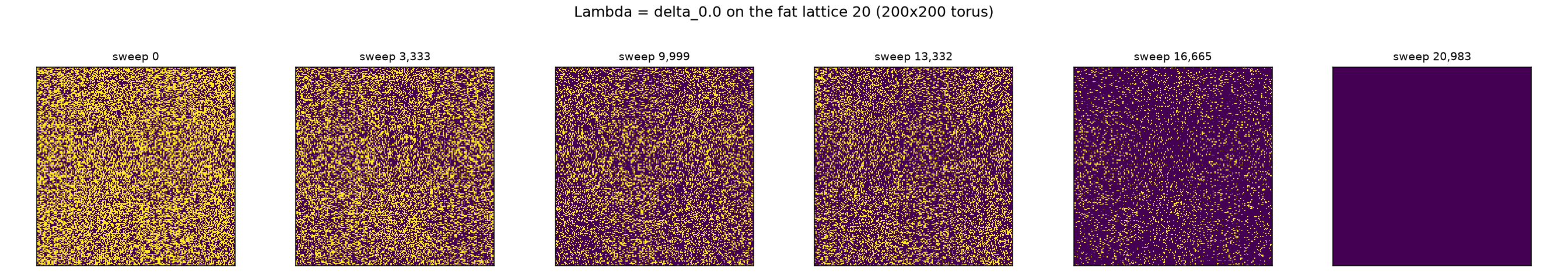}
\includegraphics[width=15cm]{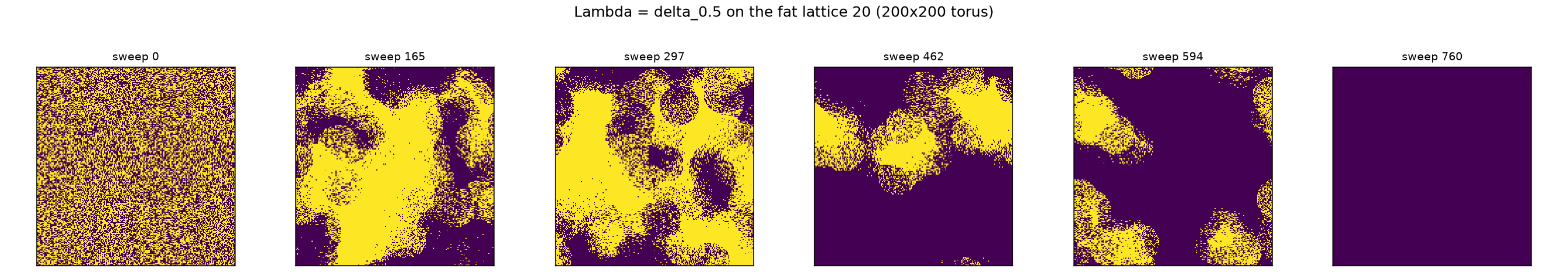}
\includegraphics[width=15cm]{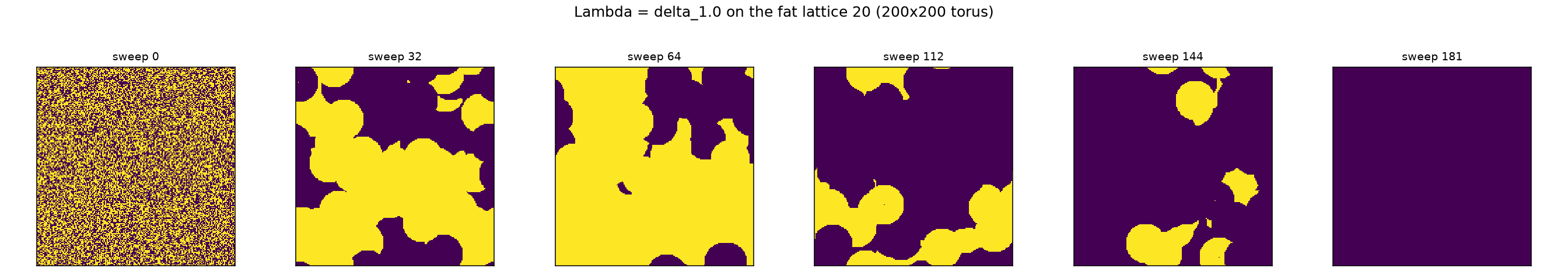}
\caption{Realizations of the classical voter model ($\Lambda=\delta_0$, top row), the
    $\Lambda=\delta_{1/2}$ broadcast voter model (middle row), and the $\Lambda=\delta_1$
    broadcast voter model (bottom row) on the \emph{$20$-fat torus of size $200$}: the graph
    obtained from the discrete torus $(\mathbb{Z}/200\mathbb{Z})^2$ by adding an edge between
    every pair of vertices at Euclidean distance less than $20$. All three simulations start
    from the same i.i.d.\ Bernoulli$(1/2)$ configuration.
    \label{fig:fat-lattice-realizations}} 
    \label{Model}
\end{figure*}

On a seemingly unrelated front, the Kingman coalescent and its generalizations occupy a
central place in the study of genealogies. While the Kingman coalescent only allows two
ancestral lines to merge at a time, Pitman and Sagitov independently observed that this
restriction can be relaxed: a finite measure $\Lambda$ on $[0,1]$ may be used to encode
simultaneous mergers of several lines at once, through a paintbox construction in which a
parameter $q$ is sampled in a certain way depending on $\Lambda$ and each line
present is included in the merger independently with probability $q$. The resulting
family of $\Lambda$-coalescents recovers the Kingman coalescent when $\Lambda=\delta_0$
and otherwise allows for multiple, simultaneous collisions; we refer the reader to
\cite{Pitman99,Berestycki09} for background on this theory.

This paper introduces and studies a variant of the voter model that combines these two
ingredients. Informally, we allow a voter to broadcast its opinion loudly to several of
its neighbors at once, instead of convincing only a single one of them at a time. 
A broadcast event at a
vertex $x$ is governed by a parameter $q\in[0,1]$, 
sampled in a certain way depending on a fixed finite \emph{broadcast measure} $\Lambda$ on $[0,1]$; 
when this happens, each neighbor of $x$ independently adopts the opinion of $x$ with probability $q$. 
We call the resulting process the \emph{$\Lambda$-broadcast voter model} 
(formally defined in Section~\ref{s:model_results} through its Markov pre-generator, and constructed graphically in Section~\ref{ss_alternative_description}). 
Taking $\Lambda=\delta_0$ forces exactly one neighbor to be affected at a time, and one recovers
the classical voter model. In this sense, the relationship between the voter and
the broadcast voter models mirrors that between the Kingman and the
$\Lambda$-coalescents: the measure $\Lambda$ simultaneously governs the rate at which
opinions are broadcast and the degree of coordination among the neighbors that are
convinced.

Related mechanisms involving simultaneous reproduction events have been studied in spatial population models. In particular, Barton, Etheridge and Véber~\cite{BartonEtheridgeVeber2010} introduced the spatial~\(\Lambda\)-Fleming–Viot process, in which a reproduction event replaces a randomly selected fraction of the population in a spatial region with offspring of a common parent. Saadi~\cite{Saadi2014} subsequently studied a closely related version in which the parent is selected at the center of the affected region. The latter mechanism is particularly close to our model, since the broadcasting vertex is also distinguished in advance. Our setting, however, concerns binary opinions on a graph, with a general finite measure~\(\Lambda\) governing the broadcast events.

Our main object of study in this work is the set of stationary measures of the $\Lambda$-broadcast
voter model. This continues a program the first author has pursued for other variants of
the voter model, most closely in \cite{Astoquillca26}, where the stationary measures of two
different generalizations of the classical dynamics are characterized through the same
duality philosophy. The present paper treats a third variant, characterized by broadcast events that allow several neighboring vertices to adopt the opinion of a common source simultaneously.

The main new feature of the broadcast voter model is the structure of its dual. Because
several neighbors of a broadcasting vertex may be affected simultaneously, the natural
dual is no longer a system of ordinary coalescing random walks, but a system of
\emph{coordinated coalescing random walks}: particles that occupy neighboring positions
of a common vertex may jump to it simultaneously and coalesce there, while, marginally,
each particle still evolves as a continuous-time simple random walk. 
We establish a moment duality between the broadcast voter model and this coordinated system 
(Lemma~\ref{lemma:duality}), and show that, despite the interaction it introduces, the coordinated system
retains two properties familiar from the classical coalescing picture: its single-particle
marginals are simple random walks, and the family is projectively consistent, in the
sense that the law of any subcollection of particles agrees with that of the
corresponding lower-dimensional system (Lemma~\ref{lemma:CCRW_properties}). This last property is reminiscent of
the sampling consistency of $\Lambda$-coalescents, and more generally of the
correspondence between exchangeable sequences of processes and coordinated processes
described in \cite{GonKurtNunezPeres2025,GonKurtTobias2021}: removing particles from
the dual system does not alter the way the remaining ones respond to a broadcast event,
exactly as removing lines from an exchangeable family does not alter the law of the
lines that remain.

Spatial extensions of~\(\Lambda\)-coalescents have also been investigated by Limic and Sturm~\cite{LimicSturm2006}. In their model, ancestral lines migrate through a geographical space and can merge only when occupying the same site. By contrast, our coordinated random walks may jump simultaneously from distinct neighboring vertices to a common vertex, where they coalesce. This distinction highlights the particular spatial interaction induced by the broadcast mechanism.

\subsection{Organization of the paper and proof strategies}
Our main results can be described as follows. First, we show in Theorem~\ref{thm:existence} that our model is well-defined on general graphs with bounded degrees. 
Then in Theorem~\ref{thm:classification_general} we construct and characterize all stationary measures of the model in terms of collision properties of the underlying simple random walks, obtaining in Corollary~\ref{c:long_lattice} a particularly simple description in the case of $\Z^d$ with $R$-range edges, $R \geq 1$. Turning our attention to the question of how the stationary measures depend on the broadcast measure, we discover that they depend on $\Lambda$ solely through a certain vector of normalized moments (Lemma~\ref{lemm:m_imply_mu}) that completely characterize the dynamics up to a linear time-change.
Additionally, in the finite-collision regime, we obtain criteria to show that distinct moment profiles imply distinct stationary measures, allowing a complete characterization in two different classes of graphs: trees with bounded degrees and $\Z^d$, $d \geq 3$ (cf.~Theorems~\ref{thm:distinctness_trees} and \ref{thm:distinctness_Zd}).
Finally, we show in Theorem~\ref{thm:continuity} that the equilibrium measures are continuous with respect to the moment profiles, and thus in particular with respect to the (normalized) broadcast measure $\Lambda$. The proof of all of these theorems hinge on a coupling we construct between systems of independent random walks and the system of coordinated coalescing random walks, allowing to reduce via duality many arguments to computations involving a single or a pair of independent random walks.

The rest of the paper is organized as follows.
In Section~\ref{s:model_results}, 
we define the $\Lambda$-broadcast voter model via its Markov pregenerator 
and state our assumptions and main results in detail. 
In Section~\ref{s:TheBVModel}, we prove Theorem~\ref{thm:existence} concerning existence of the dynamics,
provide a graphical Poissonian construction of the system, introduce the system of coordinated coalescing random walks and establish the duality between the two processes. 
Section~\ref{s:stat_meas}
is devoted to the proof of the main classification result, Theorem~\ref{thm:classification_general},
and its Corollary~\ref{c:long_lattice}, treating separately the infinite- and finite-collision regimes in Sections~\ref{ss:part2_thm_sm} and~\ref{ss:part3_thm_sm} respectively.
Section~\ref{ss_dependence_Lambda} studies the dependence of
the stationary measures on the broadcasting measure $\Lambda$, 
including the proof of the distinctness results, Theorems~\ref{thm:distinctness_trees} and~\ref{thm:distinctness_Zd}, in Section~\ref{ss_distinctness},
and the proof of the continuity result, Theorem~\ref{thm:continuity},
in Section~\ref{ss_continuity}.

\section{Model, assumptions and main results}
\label{s:model_results}

We consider graphs~$G = (V,E)$ that consist of non-empty, finite or countably infinite vertex sets~$V$ and edge sets~$E$. All graphs are assumed to be simple (no self-loops or multi-edges), undirected, locally finite and connected. 
Given a set or event $A$, we denote by $\1_A$ its indicator function and by $|A|$ its cardinality.
We write~$x \sim y$ when~$\{x,y\} \in E$ and call~$x$ and~$y$ neighbors.
Given $x \in V$, we write
\begin{equation}
\label{def:degrees}
N_x := \left\{y \in V \colon y \sim x  \right\}, \qquad \deg(x) := |N_x|
\end{equation}
to denote respectively the set of neighbors and the degree of $x$.
For our definition of the model, we will also assume that $G$ has bounded degrees, i.e.:

\medskip
\medskip
\noindent
{\bf Assumption (BD):}
\vspace{-14pt}
\begin{equation}
\label{def:Delta}
\hspace{-10pt}
\Delta := \sup_{x \in V} \deg(x) < \infty.
\end{equation}

\medskip

Let $\Lambda$ be a (non-negative, non-trivial) finite measure on $[0,1]$, which we call the \emph{broadcast measure}. 
Given a graph $G$ with bounded degrees, 
the~$\Lambda$-\emph{broadcast voter model} on $G$, 
denoted by~$\eta_t, t \ge 0$, is the Markov process with state space~$\{0,1\}^V$ 
and Markov pre-generator
\begin{equation}\label{Markov_pre_dyn}
\mathcal L_\Lambda f(\eta) = \sum_{x \in V} \sum_{ \substack{A \subseteq N_x \\ A \neq \varnothing } }  c_\Lambda(x,|A|) \cdot \big[ f(\eta^{x \to A}) - f(\eta) \big], \quad f:\{0,1\}^V \to \R \text{ local},    
\end{equation}
where
$$ \eta^{x \to A}(z) = \left\{ \begin{array}{cc}
    \eta(x), & z \in A,  \\[2mm]
    \eta(z), & z \notin A,
\end{array} \right. \qquad z \in V,  $$
and
\begin{equation}\label{def_c_Lambda}
c_\Lambda(x,k) = \int_{[0,1]} q^k(1-q)^{|N_x|-k} \; q^{-1} \Lambda(\mathrm{d}q), \qquad k \ge 1.
\end{equation}
The name is motivated by the following interpretation of the dynamics: each site~$x$ (a voter) broadcasts its opinion~$\eta_t(x) \in \{0,1\}$ simultaneously to a non-empty subset~$A$ of its neighbors~$N_x$ at rate~$c_\Lambda(x,|A|)$. 

For readers familiar with~$\Lambda$-coalescents, 
the form of the rates~$c_\Lambda(x,k)$ in \eqref{def_c_Lambda} 
can be viewed as a local version of the usual merger rates. At a broadcast from~$x$, one may think of the opinion carried by~$x$ as already participating in a merger with the opinions of the vertices selected from~$N_x$. Thus, among the~$|N_x|+1$ possible participants consisting of~$x$ and its neighbors, the vertex~$x$ is distinguished and is always present, while only the neighbors are randomly selected. This explains the appearance of~$q^{-1}$ in the broadcast rates, in place of the~$q^{-2}$ factor in Pitman's representation of the multiple-merger rates~\cite{Pitman99}.

Our first result shows that, under Assumption (BD), our model is a well-defined interacting particle system on $\{0,1\}^V$. We refer to \cite[Chapter~1]{Liggett85} for further information.

\begin{theorem}\label{thm:existence}
If $G$ has bounded degrees, then the closure of $\mathcal{L}_{\Lambda}$ in \eqref{Markov_pre_dyn} is the generator of a Feller semigroup $(S^\Lambda_t)_{t \geq 0}$. In particular, the $\Lambda$-broadcast voter model on $G$ is well defined.
\end{theorem}

The proof of Theorem~\ref{thm:existence} will be given in Section~\ref{ss:proof_existence}.

\paragraph{The standard voter model.} 
The usual voter model is recovered by taking~$\Lambda=\delta_0$. Indeed, for every~$x$,
$c_{\delta_0}(x,1) = 1$ and $c_{\delta_0}(x,k) = 0$ for all $k \ge 2$. 
Thus, only updates involving a single neighbor occur, each at rate~$1$. Consequently, the dynamics reduce to the voter model with adoption rates~$r(x,y) = 1$ for every~$x \sim y$. 



\subsection{Main results}
\label{ss:main_results}

In light of Theorem~\ref{thm:existence}, Assumption (BD) will always be made henceforth. 
Before we can state our main results, we need to discuss some properties of random walks on $G$. To that end, let $X=(X_t)_{t \geq 0}$ denote a continuous-time simple random walk on $G$ with jump rate $\Lambda[0,1]$ to each neighbor, i.e., the Markov process on $V$ that, when at $x \in V$, jumps to a neighbor $y \sim x$ at rate~$\Lambda[0,1]$. We will denote by $P_x$ the law of $X$ started from~$X_0=x$.

Let $X,Y$ be two independent copies of the random walk defined in the previous paragraph. 
We say that the graph $G$ has the \emph{infinite collision property} if, for all $x \in V$,
\begin{equation}
\label{def:ICP}
P_{x,x} \left( \left|\{t \geq 0 \colon\, X_{t-} \neq Y_{t-} \text{ and } X_t = Y_t\}\right| = \infty \right) =1
\end{equation}
where $P_{x,y}$ denotes the joint law of $X,Y$ when started from $x,y \in V$ respectively.
And we say that $G$ has the \emph{finite collision property} if, for all $x \in V$,
\begin{equation}
\label{def:FCP}
P_{x,x} \left( \left|\{t \geq 0 \colon\, X_{t-} \neq Y_{t-} \text{ and } X_t = Y_t\}\right| < \infty \right) =1.
\end{equation}
For graphs with bounded degrees, these two alternatives form a dichotomy (cf.~\cite[Remark 2.3]{Asto2026Coll}). Let~$p(x,y):=\mathds{1}\{x \sim y\} \deg(x)^{-1}$ be the transition kernel of the embedded discrete-time simple random walk. 
Denote by
$$ \mathscr{H} = \left\{ \alpha: V \to [0,1] \quad
\text{such that} \quad \sum_{y \sim x} p(x,y)\alpha(y) = \alpha(x) \text{ for all } x \in V \right\}$$
the space of~\([0,1]\)-valued harmonic functions for simple random walk on~\(G\). 
For $\alpha \in \mathscr{H}$, let~$\nu_\alpha$ be the product probability measure on~$\{0,1\}^V$ with one-site marginals
\begin{equation}\label{def_nu_alpha}
\nu_\alpha( \{ \eta : \eta(x) = 1 \} ) = \alpha(x).
\end{equation}
We denote by~$\delta_{\overline{0}}$ and~$\delta_{\overline{1}}$ the point masses at the consensus configurations all-0 and all-1 corresponding to the constant harmonic functions~$\alpha \equiv 0$ and~$\alpha \equiv 1$, respectively.

We say that a probability measure $\mu$ on $\{0,1\}^V$ is \emph{stationary} for the $\Lambda$-broadcast voter model if~$\mu = \mu S^\Lambda_t$ for every~$t \ge 0$, and denote by~$\mathscr{I}_\Lambda$ the set of stationary measures. 
We say that~$\mu \in \mathscr{I}_\Lambda$ is \emph{extremal} if it cannot be written as a nontrivial convex combination of distinct stationary measures. 
The set of extremal stationary measures is denoted by~$\mathscr{I}_{\Lambda,e}$. 
By Theorem~1.8(c) of~\cite{Liggett85} (note that ``stationary'' is called ``invariant'' therein), \(\mathscr I_\Lambda\) is the closed convex hull of~\(\mathscr I_{\Lambda,e}\). 
Consequently,~$\mathscr I_\Lambda$ is completely characterized by $\mathscr I_{\Lambda,e}$.

Our first main result gives a complete classification of $\mathscr I_{\Lambda,e}$ in terms of the collision behaviour of simple random walks on \(G\).
\begin{theorem}
\label{thm:classification_general}
Let~\(G\) be a graph with bounded degrees
and let~\(\Lambda\) be a finite measure on~\([0,1]\).
\begin{enumerate}
    \item For each \(\alpha\in\mathscr H\), the weak limit~$  \mu_{\alpha,\Lambda} := \lim_{t\to\infty}\nu_\alpha S_t^\Lambda $ exists and is stationary.
    \item If~\(G\) has the infinite collision property, then~$\mathscr I_{\Lambda,e} = \{\delta_{\mathbf 0},\delta_{\mathbf 1}\}.$
    \item If~\(G\) has the finite collision property, then~$\mathscr I_{\Lambda,e} = \{\mu_{\alpha,\Lambda}:\alpha\in \mathscr H\}.$
\end{enumerate}
\end{theorem}
Thus, although the dual process allows coordinated multiple mergers, the qualitative classification of stationary measures is governed by the same collision dichotomy as in the classical voter model
(cf.~\cite[Chapter V, Theorem~1.8]{Liggett85}).



As a direct application, we obtain a particularly simple description for finite-range broadcast models on~\(\mathbb Z^d\), defined as follows. Let
\begin{equation}\label{def_long_lattice}
G_{d,R}=(\mathbb Z^d,E_R), \qquad \text{where} \qquad \{x,y\} \in E_R \quad\Longleftrightarrow\quad 0< \|x-y\|_1 \leq R.    
\end{equation}
Thus, in the corresponding broadcast voter model, a voter may simultaneously broadcast its opinion to vertices lying within range~\(R\).
In this case, the harmonic functions are constant:~$\alpha \equiv p,\; p \in [0,1]$, as we will show in the proof of the following.
\begin{corollary}
\label{c:long_lattice}
Fix \(R\in\mathbb N\) and let \(\Lambda\) be a finite measure on \([0,1]\). For the \(\Lambda\)-broadcast voter model on $G_{d,R}$,
\begin{enumerate}
    \item if \(d=1,2\), then $ \mathscr I_{\Lambda,e} = \{\delta_{\mathbf 0},\delta_{\mathbf 1}\}; $ 
    \item if \(d\ge3\), then
    $ \mathscr I_{\Lambda,e} = \{\mu_{p,\Lambda}:p\in[0,1]\};  $
    \item if \(d\ge3\), then each \(\mu_{p,\Lambda}\) is ergodic with respect to translations of \(\mathbb Z^d\).    
\end{enumerate}
\end{corollary}
Theorem~\ref{thm:classification_general} and Corollary~\ref{c:long_lattice} 
are proved in Section~\ref{s:stat_meas}.

\begin{remark}
Corollary 3.2 shows that, for finite-range broadcast models on \(\mathbb Z^d\), neither the broadcast range \(R\) nor the broadcasting measure \(\Lambda\) changes the qualitative classification of the extremal stationary measures. This does not mean, however, that \(\Lambda\) has no effect on the stationary laws. The dependence of the family~\(\mu_{\alpha,\Lambda}\) on the broadcasting measure will be studied in Section~\ref{ss_dependence_Lambda}.
\end{remark}

\medskip
Our second main question concerns how the stationary measures depend on the broadcasting
measure $\Lambda$ itself. First note that, on a graph of maximum degree $\Delta$, at most $\Delta$ neighbors can take part in a broadcast, and it is not difficult to show that, apart from a linear time-change that doesn't affect the stationary measures, 
the dynamics depends on $\Lambda$ only through the moments
\begin{equation}
\label{def:moments}
\lambda^{(k)}_\Lambda=\frac{\int_{[0,1]} q^{k-1}\,\Lambda(dq)}{\Lambda[0,1]},\qquad 1\le k\le\Delta.
\end{equation}
Noting that $\lambda^{(1)}_\Lambda=1$ for any $\Lambda$, denote the moment vector (or profile) by
\begin{equation}
\label{def:momentvector}
\mathsf{m}(\Lambda) := \left(\lambda^{(2)}_\Lambda, \ldots, \lambda^{(\Delta)}_\Lambda \right).
\end{equation}
In Lemma~\ref{lemm:m_imply_mu} below, 
we show that equality of the moment profiles implies equality of the stationary measures. 
We then investigate the converse, which by Theorem~\ref{thm:classification_general} is interesting only when the graph has the finite collision property. 
For general such graphs with bounded degrees, we obtain
in Lemma~\ref{lemma_FCP_distinct} below several sufficient conditions for distinct moment profiles to yield distinct stationary measures. This leads to complete characterizations on some natural classes of graphs, as stated next.
\begin{theorem}
\label{thm:distinctness_trees}
Let~\(G\) be a tree with bounded degrees and the finite collision property. Then, for every nontrivial \(\alpha\in\mathscr H\),
$$ \mu_{\alpha,\Lambda_1} = \mu_{\alpha,\Lambda_2} \quad\Longleftrightarrow\quad \mathsf{m}(\Lambda_1)=\mathsf{m}(\Lambda_2). $$    
\end{theorem}
Examples of trees satisfying the above include transient cases such as 
supercritical Galton-Watson trees with bounded degrees conditioned to survive,
but also recurrent ones such as the Comb lattice (cf.\ \cite{KrishnapurPeres2004}).

We also obtain the analogous result for the integer lattice in transient dimensions.
\begin{theorem}
\label{thm:distinctness_Zd}
Let \(d\ge3\), and consider the nearest-neighbor broadcast voter model on \(\mathbb Z^d\). 
Then, for every \(p\in(0,1)\),
$$ \mu_{p,\Lambda_1} = \mu_{p,\Lambda_2} \quad\Longleftrightarrow\quad \mathsf{m}(\Lambda_1)=\mathsf{m}(\Lambda_2). $$    
\end{theorem}
Thus, on these graph classes, the stationary measure completely determines the moment profile of the broadcasting mechanism. Whether the same equivalence holds on every graph with bounded degrees and the finite collision property remains open.

Theorems~\ref{thm:distinctness_trees} and~\ref{thm:distinctness_Zd} 
are proved in Section~\ref{ss_distinctness}.

\medskip
For our final result, we obtain continuity of the stationary measures with respect to the broadcasting measure.
\begin{theorem}
\label{thm:continuity}
Let $G$ be a graph with bounded degrees. 
Let \(\alpha\in\mathscr H\), and let \((\Lambda_n)_{n\ge1}\), \(\Lambda\) be finite measures on \([0,1]\). Then
$$ \mathsf{m}(\Lambda_n)\to \mathsf{m}(\Lambda) \quad \text{ implies } \quad \mu_{\alpha,\Lambda_n} \Longrightarrow \mu_{\alpha,\Lambda}. $$
In particular, the conclusion holds whenever \(\Lambda_n/\Lambda_n[0,1]\) converges weakly to \(\Lambda/\Lambda[0,1]\).    
\end{theorem}
Theorem~\ref{thm:continuity} is proved in Section~\ref{ss_continuity}.

\medskip
The proofs of these results rely on a single coupling idea introduced in Section~\ref{ss:part1_thm_sm}. We couple, for each $k$, the
$k$-particle coordinated coalescing random walk with a system of $k$ independent simple
random walks so that the two systems agree exactly until some pair of particles becomes
close enough to interact through a common broadcast; from that point on, the two systems
may be driven apart. Because this coupling only breaks near an interaction, its
persistence can be controlled entirely through the collision behaviour of ordinary
random walks, which allows us to transfer classical estimates for independent random
walks to the coordinated system. This coupling underpins our proofs for both the existence and the
classification of the measures $\mu_{\alpha,\Lambda}$, as well as the continuity result
described above.

\section{The broadcast voter model}
\label{s:TheBVModel}

In this section, we first we prove Theorem~\ref{thm:existence} in Section~\ref{ss:proof_existence}, 
and then we provide a graphical Poissonian construction of the model in Section~\ref{ss_alternative_description}. 
In Section~\ref{ss:CCRW}, we define and study systems of coordinated coalescing random walks, which are shown in Section~\ref{ss_duality} to be dual to the $\Lambda$-broadcast voter model.

\subsection{Proof of Theorem~\ref{thm:existence}}
\label{ss:proof_existence}

The proof will consist of two observations concerning the rates of the process. 
The first shows that the total rate of broadcasts originating from any fixed vertex is finite, with bounded degrees making them uniformly bounded. 
The second shows how Assumption (BD) allows us to verify directly a standard sufficient condition for the existence of the corresponding interacting particle system.

The total rate of nontrivial broadcasts originating from~$x$ is 
\begin{equation}\label{eq_total_broadcast_rate}
\gamma_\Lambda(x) := \sum_{ \substack{A \subseteq N_x \\ A \neq \varnothing } } c_\Lambda(x,|A|) = \int_{[0,1]} \sum_{ \substack{A \subseteq N_x \\ A \neq \varnothing } } q^{|A|-1}(1-q)^{|N_x|-|A|} \; \Lambda(\mathrm{d}q) = \int_{[0,1]} \phi_{|N_x|}(q)  \; \Lambda(\mathrm{d}q),
\end{equation}
where
$$ \phi_n(q) = \left\{ \begin{array}{ll}
    \frac{1-(1-q)^n}{q}, & q > 0 \\[2mm]
    n, & q= 0, 
\end{array} \right.  $$
Using the finite geometric series identity
\begin{equation}\label{geometric_serires}
    \frac{1-(1-q)^n}{q} = \sum^{n-1}_{j = 0} (1-q)^j \leq n, \qquad q \in (0,1],
\end{equation}
we obtain
$$ \gamma_\Lambda(x) \leq \Lambda[0,1] \cdot |N_x|. $$
In particular, every vertex has a finite effective broadcast rate, 
and if the degrees are uniformly bounded then so are the rates.

To justify the existence of the process, we use the general construction of interacting particle systems with simultaneous updates from~\cite[Chapter~1, Theorem~3.9]{Liggett85} (note that this is not a spin system). 
In the notation of that result, \(c_T(\eta,\cdot)\) denotes the transition-rate kernel associated with an update of the finite set \(T\), and
$$ C_T:=\sup_\eta c_T(\eta,\{0,1\}^T) $$
its maximal total rate. For the~$\Lambda$-broadcast voter model,
$$ C_T=\sum_{\substack{x\in V\\T\subseteq N_x}}c_\Lambda(x,|T|). $$
Hence, rearranging the sums,
\begin{align*}
\sum_{T\ni z} C_T = \sum_{x \sim z} \sum_{ \substack{ T \subseteq N_x \\ z \in T } } c_\Lambda(x,|T|) & = \sum_{ x \sim z } \sum_{ \substack{T \subseteq N_x \\ x \in T } } \int_{[0,1]} q^{|T|-1}(1-q)^{|N_x|-|T|} \; \Lambda(\mathrm{d}q)\\[2mm]
& = \sum_{ x \sim z } \int_{[0,1]} \sum_{ B \subseteq N_x \setminus \{z\} } q^{|B|}(1-q)^{|N_x \setminus \{z\}|-|B|} \; \Lambda(\mathrm{d}q) \\[2mm] 
& = \Lambda[0,1] \cdot |N_z|, 
\end{align*}
where the last equality follows from the binomial identity. 
Thus~$\sup_{z\in V}\sum_{T\ni z}c_T < \infty $, i.e., condition~(3.3) of~\cite[Chapter~I]{Liggett85} is satisfied. Moreover, changing the opinion at a site~\(u\) can affect~$c_T(\eta, \cdot)$ only when~\(T\subseteq  N_u  \), implying that $c_T(u)$ as defined in \cite[Chapter~I]{Liggett85} is bounded by $2 C_T \1_{\{T \subseteq N_u\}}$. Thus, by the computation above, the parameter $M$ from \cite[Chapter~I, Theorem~3.9]{Liggett85} is bounded by $2 \Lambda[0,1] \max_{z \in V} |N_z|^2 < \infty$, verifying the dependence condition (3.8). Thus the hypotheses of~\cite[Chapter~I, Theorem~3.9]{Liggett85} are satisfied, and the \(\Lambda\)-broadcast voter model is well-defined.
\hfill $\qed$

\subsection{Poissonian description of the dynamics}\label{ss_alternative_description}
Motivated by the Poissonian construction of the~$\Lambda$-coalescent, we give an alternative description of the broadcast voter dynamics that clarifies the particular form of the rates~\(c_\Lambda\) from~\eqref{Markov_pre_dyn}. For readers familiar with~\(\Lambda\)-coalescents, this shows how the usual~\(q\)-merger mechanism naturally adapts to the present setting; for readers unfamiliar with this theory, it provides a self-contained account of the underlying probabilistic mechanism. We refer again to~\cite{Berestycki09} for background on~\(\Lambda\)-coalescents and their Poissonian construction.

Recall the total rate of broadcasts~$\gamma_\Lambda$ from~\eqref{eq_total_broadcast_rate} and the function~$\phi_n$. Since~$G$ is connected (with at least two vertices) and~$\Lambda$ is non-trivial, we have~$\gamma_\Lambda(x)>0$ for every~$x\in V$. Define the probability measure
\begin{equation}\label{eq_tilted_Lambda}
\widehat\Lambda_x(\mathrm dq) := \frac{\phi_{|N_x|}(q)}{\gamma_\Lambda(x)} \; \Lambda(\mathrm dq), \qquad x \in V.
\end{equation}
The broadcasts originating from~$x$ may then be generated as follows. Independently for each~$x \in V$, let a Poisson clock ring at rate~$\gamma_\Lambda(x)$. At each ring, independently of the past, sample a parameter~$q$ according to~$\widehat\Lambda_x$; 
\begin{itemize}
\item if~$q>0$, select each vertex of~$N_x$ independently with probability~\(q\), conditional on selecting at least one vertex;
\item if~$q=0$, select one vertex of~$N_x$ uniformly at random. 
\end{itemize}
All selected vertices then simultaneously adopt the opinion of~\(x\).

The rule at~\(q=0\) is the natural limiting case of the construction for~\(q>0\). Indeed, as~\(q\downarrow0\), a collection of independent Bernoulli-\(q\) selections, conditioned on being nonempty, consists with probability tending to one of a single uniformly chosen neighbor.

To verify that this description agrees with~\eqref{Markov_pre_dyn}, fix a nonempty set~\(A\subseteq N_x\). Conditional on a sampled value~\(q>0\), the probability that precisely the vertices in~\(A\) are selected is
\[ \frac{q^{|A|}(1-q)^{|N_x|-|A|}}{1-(1-q)^{|N_x|}}. \]
Hence the contribution from~$q>0$ to the rate of broadcasts from~$x$ precisely to~$A$ is
$$ \gamma_\Lambda(x) \cdot \int_{(0,1]} \frac{q^{|A|}(1-q)^{|N_x|-|A|}} {1-(1-q)^{|N_x|}} \; \widehat\Lambda_x(\mathrm dq) = \int_{(0,1]}
q^{|A|-1}(1-q)^{|N_x|-|A|} \;\Lambda(\mathrm dq).$$
When~$q=0$, only singleton sets can be selected. 
If~$A=\{y\}$, its rate coming from the possible atom at zero is
\[ \gamma_\Lambda(x) \cdot \widehat\Lambda_x({0}) \cdot \frac{1}{|N_x|}
= \Lambda({0}), \]
whereas sets with~$|A| \ge 2$ receive no contribution. Combining the two cases gives
\[ c_\Lambda(x,|A|) = \int_{[0,1]} q^{|A|-1}(1-q)^{|N_x|-|A|} \; \Lambda(\mathrm dq),\]
as required.

This description also gives a natural graphical representation of the process. At a broadcast event from~\(x\) to~\(A\), draw arrows from~\(x\) to every vertex of~\(A\). Read forward in time, all vertices in~\(A\) simultaneously adopt the opinion of~\(x\). Read backward in time, a path located at a vertex of~\(A\) follows the corresponding arrow to~\(x\). Hence, several backward paths may jump simultaneously to the same vertex and coalesce. This graphical picture motivates the dual process introduced in the next subsection.

\subsection{Coordinated coalescing random walks}
\label{ss:CCRW}
As is standard in the study of voter models, we describe the flow of opinions backwards in time by means of coalescing random walks. In the present setting, the graphical representation above leads to a system in which several walks may jump simultaneously in response to the same broadcast event. We call this process \emph{coordinated coalescing random walks}. We now formalize its dynamics, beginning with some notation.

Given $z,x \in V$ and~$A \subseteq V$, we define 
\begin{equation*}
     z^{A \to x} := \left\{ \begin{array}{cc}
        x & \text{if } z \in A, \\[2mm]
        z & \text{otherwise}
    \end{array} \right. 
\end{equation*}
In words, we move~$z$ to~$x$ only if~$z \in A$. If~$A = \{y\}$, we simply write~$z^{y \to x}$. Moreover, for any integer~$k \in \N$ and vector $\mathbf{y} = (y_1,\dots, y_k) \in V^k$ we write~$ \mathbf{y}^{A \to x} := (y^{A \to x}_1, \dots, y^{A \to x}_k).$ Thus,~$\mathbf y^{B\to x}$ is obtained by moving to~$x$ all coordinates whose current location belongs to~$B$, leaving the remaining coordinates unchanged.

Given~$k \in \N$ and a finite measure~$\Lambda$ on~$[0,1]$, we define the~$k$-particle~$\Lambda$-\emph{coordinated coalescing random walks} ($\Lambda$-CCRWs) as the continuous-time Markov chain~$(Y^1_t, \dots, Y^k_t)_{t \ge 0}$ with state space~$V^k$ and generator
\begin{equation}\label{gen_Lambda_CCRW}
\mathcal G^{(k)}_\Lambda f(\mathbf y) = \sum_{x \in V} \sum_{ \substack{ B \subseteq N_x \cap S(\mathbf y) \\ B \neq \varnothing } } r_\Lambda( \mathbf y,x,B ) \cdot \big[ f(\mathbf y^{B \to x}) - f(\mathbf y) \big], \quad \mathbf y = (y_1,\dots,y_k),
\end{equation}
where~$S(\mathbf y) := \{y_1,\dots,y_k\}$ denotes the set of occupied sites, and
\begin{equation}\label{eq_CCRW_rate}
r_\Lambda( \mathbf y,x,B ) := \int_{[0,1]} q^{|B|}(1-q)^{|N_x \cap S(\mathbf y)| - |B|} \; q^{-1} \Lambda(\mathrm d q).
\end{equation}

The following calculation justifies the choice of the rates in~\eqref{eq_CCRW_rate} by relating them directly to the broadcast dynamics. Fix a site~$x \in V$, a vector~$\mathbf y \in V^k$, and a nonempty set~$B \subseteq N_x \cap S(\mathbf y)$. Then
\begin{align}\label{c_Lambda_r_Lambda}
\sum_{\substack{A\subseteq N_x  \\ A \cap S(\mathbf y)=B}}
c_\Lambda(x,|A|) & = \sum_{C\subseteq N_x \setminus S(\mathbf y)}
c_\Lambda(x,|B|+|C|) \nonumber \\[2mm]
&= \int_{[0,1]} q^{|B|-1} (1-q)^{|N_x\cap S(\mathbf y)|-|B|}
\hspace{-10pt}
\sum_{C\subseteq N_x \setminus S(\mathbf y)}
q^{|C|} (1-q)^{|N_x\setminus S(\mathbf y)|-|C|}
\; \Lambda(\mathrm dq) \nonumber \\[2mm]
& = r_\Lambda(\mathbf y,x,B), 
\end{align}
where the last equality follows from the binomial identity.

The condition~\(A\cap Y(\mathbf y)=B\) means that, among the sites occupied by the~\(\Lambda\)-CCRWs, precisely those in~\(B\) are reached by the broadcast from~\(x\). Whether the broadcast also reaches neighbors not occupied by the walks is irrelevant to their evolution, and all such possibilities are accounted for in the sum above. Consequently, a transition~\(\mathbf y\to\mathbf y^{B\to x}\) occurs whenever a broadcast from~\(x\) reaches precisely the occupied neighboring sites in~\(B\). All walks at these sites then jump simultaneously to~\(x\) and coalesce there. In the case~\(|B|=1\), the transition corresponds to a jump of a single walk; if the destination~\(x\) is already occupied by another walk, the two coalesce exactly as in the standard system of coalescing random walks.

Having clarified the structure of the transition rates, we now record some basic properties of the~$\Lambda$-CCRWs.
\begin{lemma}\label{lemma:CCRW_properties}
Let~$G$ be a graph with bounded degrees. The family of~$\Lambda$-CCRWs satisfies the following properties.
\begin{enumerate}
    \item For every~$k \in \N$, the~$k$-particle process is nonexplosive.
    \item The family is projectively consistent: any subcollection of coordinates of a~$k$-particle~$\Lambda$-CCRW evolves as the corresponding lower-dimensional~$\Lambda$-CCRW.
    \item Each coordinate is a continuous-time simple random walk on~$G$, with jump rate~$\Lambda[0,1]$ along each incident edge.
\end{enumerate}
\end{lemma}
\begin{proof}
Fix a vector~$\mathbf y \in V^k$. The total rate of transitions associated with broadcasts originating from \(x\) with~$N_x \cap S(\mathbf y) \neq \varnothing$ is
$$ \sum_{ \substack{ B \subseteq N_x \cap S(\mathbf y) \\ B \neq \varnothing } } r_\Lambda( \mathbf y,x,B ) = \int_{[0,1]} \phi_{|N_x \cap S(\mathbf y)|}(q) \; \Lambda(\mathrm{d}q) \stackrel{\eqref{geometric_serires}} \leq \Lambda[0,1] \cdot |N_x \cap S(\mathbf y)|, $$
where the equality follows from the same calculation done in~\eqref{eq_total_broadcast_rate}. Summing over~$x$, we have that the total rate from~$\mathbf y$ is bounded by
$$ \sum_{x \in V} \Lambda[0,1] \cdot |N_x \cap S(\mathbf y)| = \sum_{x \in S(\mathbf y)} \Lambda[0,1] \cdot |N_x| \leq k \Lambda[0,1] \cdot \sup_{x \in V} \deg(x).  $$
Therefore, by Theorem~2.7.1(ii) of~\cite{Norris97}, the \(k\)-particle \(\Lambda\)-CCRW is nonexplosive. 

Projective consistency follows naturally from the graphical construction. The broadcast marks act on sites rather than on labelled walks: removing some of the walks does not alter how the remaining ones respond to any broadcast event. Hence any subcollection of coordinates is governed by the same graphical construction and has the law of the corresponding lower-dimensional~\(\Lambda\)-CCRW. This can also be seen directly from the transition rates. Let~$\mathbf y' \in V^m$,~$m \leq k$, be obtained by retaining some of the coordinates of~\(\mathbf y\), and fix a site~$x$ and nonempty set~$B\subseteq N_x\cap S(\mathbf y').$ 
In the full system of $k$ coordinated coalescing random walks, a transition inducing~\(\mathbf y'\to(\mathbf y')^{B\to x}\) may additionally move any subset
$$ C\subseteq N_x\cap\bigl(S(\mathbf y)\setminus S(\mathbf y')\bigr). $$
Therefore its total rate is
\begin{align*}
\sum_{C\subseteq N_x\cap(S(\mathbf y)\setminus S(\mathbf y'))} r_\Lambda(\mathbf y,x,B\cup C) &= \int_{[0,1]} q^{|B|-1} (1-q)^{|N_x\cap S(\mathbf y')|-|B|} \\ &\quad\times 
\hspace{-10pt}
\sum_{C\subseteq N_x\cap(S(\mathbf y)\setminus S(\mathbf y'))} q^{|C|} (1-q)^{|N_x\cap(S(\mathbf y)\setminus S(\mathbf y'))|-|C|} \,\Lambda(\mathrm dq) \\ &= r_\Lambda(\mathbf y',x,B), \end{align*} 
where the last equality follows from the binomial identity. This proves projective consistency.

Finally, by projective consistency, it is enough to identify the one-particle process. Suppose that the walk is at \(y\). For every neighbor \(x\sim y\), the only possible nonempty subset of~$N_x\cap\{y\} $ is~\(B=\{y\}\). Hence
$$ r_\Lambda(y,x,\{y\}) = \int_{[0,1]} q^{0}(1-q)^0\,\Lambda(\mathrm dq) = \Lambda[0,1]. $$
Thus the one-particle process jumps from \(y\) to each of its neighbors at rate \(\Lambda[0,1]\), and is therefore a continuous-time simple random walk on \(G\), run at speed \(\Lambda[0,1]\). By projective consistency, each coordinate of the \(k\)-particle \(\Lambda\)-CCRW has the same marginal law. We have completed the proof.
\end{proof}

The projective consistency obtained above is reminiscent of the sampling consistency of~\(\Lambda\)-coalescents: restricting the system to a smaller collection of walks leaves the corresponding dynamics unchanged. At the level of the rates, both properties arise from the same mechanism: an additional particle may either participate in an event or not, and summing over these possibilities recovers the rate of the restricted system.

The single marginal statement also clarifies the relation with ordinary coalescing random walks. Each individual trajectory is a continuous-time simple random walk, and the usual coalescing mechanism is still present: a single walk may jump onto a site already occupied by another walk, in which case they coalesce. The broadcast dynamics add a second mechanism: several walks occupying neighbors of the same vertex may jump simultaneously to that vertex and coalesce in one coordinated event.



\subsection{Duality}\label{ss_duality}
The graphical representation introduced above suggests that the~$\Lambda$-CCRWs are the natural dual process of the broadcast voter model: when a broadcast is read backward in time, every walk located at a vertex reached by the broadcast jumps to the broadcasting vertex. In order to state such a relationship, we introduce some additional notation.

Given~$\mathbf{y} \in  V^k$ and~$\eta \in \{0,1\}^V$, we denote by~$P_{\mathbf y}$ (resp.~$P_\eta$) a probability measure under which~$(Y^1_t,\dots,Y^k_t)_t$ (resp.~$(\eta_t)_t)$ is defined and satisfies~$P_{\mathbf{y}}( (Y^1_0,\dots,Y^k_0)  = \mathbf{y} ) = 1$ (resp.~$P_\eta( \eta_0 = \eta )=1$ ). The expectation operator associated with this measure is denoted by~$E_{\mathbf y}$ (resp.~$E_\eta$).

\begin{lemma}\label{lemma:duality}
Let~$G$ be a graph with bounded degrees
Given $k \in \N$ and a finite measure~$\Lambda$ on~$[0,1]$, the~$\Lambda$-broadcast voter model~$(\eta_t)_t$ and the~$k$-particle~$\Lambda$-CCRW~$(Y^1_t,\dots,Y^k_t)_t$ are dual with respect to the function
\begin{equation}\label{Def_D_k_duality}
    D^{(k)}(\eta,\mathbf y) := \eta(y_1) \cdots  \eta(y_k), \quad \eta \in \{0,1\}^V, \; \mathbf y = (y_1,\dots,y_k) \in V^k,
\end{equation}
that is, for any configuration~$\eta \in \{0,1\}^V$ and vector~$\mathbf{y} \in V^k$,
\begin{equation}\label{duality_equality_raw}
E_\eta D^{(k)}(\eta_t,\mathbf y) = E_\mathbf{y} D^{(k)}( \eta, (Y^1_t,\dots,Y^k_t) ).
\end{equation}
\end{lemma}
\begin{proof}
We write~$D = D^{(k)}$ and~$D_{\eta}(\mathbf{y}) = D_{\mathbf{y}}(\eta) =D(\eta,\mathbf{y})$. Theorem~3.42 of~\cite{Liggett2010} states that a sufficient condition for~\eqref{duality_equality_raw} is to prove that
\begin{equation}\label{duality_to_proof}
\mathcal L_\Lambda D_{\mathbf{y}}(\eta) = \mathcal G_\Lambda^{(k)} D_\eta(\mathbf y) \quad \text{for every } \eta \in \{0,1\}^V \text{ and } \mathbf y \in V^k,
\end{equation}
where~$\mathcal L_\Lambda$ is the Markov pre-generator of the broadcast voter model~\eqref{Markov_pre_dyn} and~$\mathcal G^{(k)}_\Lambda$ is the generator~\eqref{gen_Lambda_CCRW}.
To prove \eqref{duality_to_proof},
we use the following observation:
\[ \eta^{x \to A}(z) = \eta(z^{A \to x}) \quad \text{for any } x,z \in V \text{ and } \eta \in \{0,1\}^V . \]
Fix a vertex~$x$, a vector~$\mathbf y = (y_1,\dots,y_k)$ and a configuration~$\eta$. Then, for any~$A \subseteq N_x$,
\begin{equation*}
D_\mathbf{y}(\eta^{x \to A}) = D_\eta(\mathbf{y}^{A \to x}) = D_\eta( \mathbf y^{A \cap S(\mathbf y) \to x} ).
\end{equation*}  
Hence,
\begin{align*}
\sum_{ \substack{ A \subseteq N_x \\ A \neq \varnothing } }c_\Lambda(x,|A|) & \cdot \big( D_\mathbf{y}(  \eta^{x \to A}) )  - D_\mathbf{y}(\eta) \big) \\[2mm]
& = \sum_{ \substack{A \subseteq N_x \\ S(\mathbf y) \cap A \neq \varnothing }}c_\Lambda(x,|A|) \cdot \big( D_\eta(\mathbf y^{A \cap S(\mathbf y) \to x}) ) - D_\eta (\mathbf y) \big) \\[2mm]
& = \sum_{ \substack{B \subseteq N_x \cap S(\mathbf y) \\ B \neq \varnothing }} \sum_{ \substack{A \subseteq N_x \\ A \cap S(\mathbf y) = B } } c_\Lambda(x,|A|) \cdot \big( D_\eta(\mathbf{y}^{B \to x}) ) - D_\eta(\mathbf y) \big) \\
& \hspace{-1.2mm} \stackrel{\eqref{c_Lambda_r_Lambda}} = \sum_{ \substack{B \subseteq N_x \cap S(\mathbf y) \\ B \neq \varnothing }} r_\Lambda(\mathbf y,x,B) \cdot \big( D_\eta(\mathbf{y}^{B \to x}) ) - D_\eta(\mathbf y) \big), 
\end{align*}
Finally, summing over~$x \in V$, we have~\eqref{duality_to_proof}. This completes the proof.
\end{proof}

\section{Stationary measures}
\label{s:stat_meas}
In this section, we use the duality developed in Section~\ref{ss_duality} to characterize the stationary measures of the \(\Lambda\)-broadcast voter model, 
proving Theorem~\ref{thm:classification_general} and Corollary~\ref{c:long_lattice}.
As in the classical voter model, the description is determined by the collision behavior of the random walks underlying the dual process. In Section~\ref{ss:part1_thm_sm}, we construct a coupling between systems of coordinated coalescing random walks and independent random walks and use it to construct a family of stationary measures indexed by bounded harmonic functions, proving part $1.$ of Theorem~\ref{thm:classification_general}. 
The coupling constructed in this section will be used throughout of the paper. 
Sections~\ref{ss:part2_thm_sm} and~\ref{ss:part3_thm_sm} then characterize the extremal stationary measures under the infinite and finite collision properties, proving parts $2.$ and $3.$ of Theorem~\ref{thm:classification_general} respectively.
The proof of Corollary~\ref{c:long_lattice} is given in Section~\ref{ss:proof_corollary}.

\subsection{Construction of the stationary measures}\label{ss:part1_thm_sm}
In this section, we prove the first assertion of Theorem~\ref{thm:classification_general}. The main tool is a coupling between the~\(\Lambda\)-CCRW and a system of independent random walks. Besides yielding the existence of the measures~\(\mu_{\alpha,\Lambda}\), this coupling will also be used in the proofs of the two collision regimes. 

For~$k \in \N$, let~$(\mathbf X_t)_{t \ge 0}$ with
$$ \mathbf X_t = (X^1_t,\dots, X^k_t) $$
be a system of~\(k\) independent continuous-time simple random walks on~\(G\), each jumping along every incident edge at rate~\(\Lambda[0,1]\). Its generator is
\begin{equation}\label{matrix_q_irw}
\mathcal G^{(k)}_{\Lambda,\mathrm{irw}} f( \mathbf x )  =  \sum^k_{i=1} \sum_{y \sim x_i} \Lambda[0,1] \cdot \left( f( \mathbf x^{y \to i}) - f( \mathbf x ) \right), \quad f:V^k \to \R
\end{equation}
where
$$ \mathbf x = (x_1,\dots,x_k) \qquad \text{and} \qquad \mathbf x^{y \to i} = (x_1, \dots, x_{i-1},y,x_{i+1},\dots,x_k). $$
By Lemma 2.1, each coordinate of the~\(k\)-particle~\(\Lambda\)-CCRW has precisely the same marginal dynamics as the corresponding coordinate of~\((\mathbf X_t)_{t\ge0}\). The difference between the two systems only becomes visible when two or more particles can interact through the same broadcast event. We now make this comparison precise. \\[2mm]

\noindent \textbf{Coupling.} Fix~$k \in \N$. We couple a system~\((\widehat{\mathbf X}_t)_{t\ge0}\) of~$k$ independent random walks with a~\(k\)-particle \(\Lambda\)-CCRW~\((\widehat{\mathbf Y}_t)_{t\ge0}\), both started from the same vector~\(\mathbf z\in V^k\), whose entries are initially assumed to be distinct. In fact, our coupling will allow the construction of all of these systems at the same time. 

Given a graph $G$ with bounded degrees and a finite measure $\Lambda$ on $[0,1]$,
recall the graphical representation from Section~\ref{ss_alternative_description} and denote by
$\mathscr{G}_{\Lambda}$ the system of arrows described therein.
Then, any system of $\Lambda$-CCRWs may be constructed as a deterministic function of $\mathscr{G}_{\Lambda}$.
Further, on suitable probability space with probability measure $\widehat{\mathbb P}$ and expectation operator~$\widehat{\E}$, 
let $\mathscr{G}^{(n)}_{\Lambda}$ be an i.i.d.\ sequence of arrow systems, each distributed as $\mathscr{G}_\Lambda$.
Given $k \in \N$ and distinct initial positions $\mathbf{z} = (z_1, \ldots, z_k) \in V^k$,
we define the process
$\widehat{\mathbf{X}}_t = (\widehat{X}^{1}_t, \ldots, \widehat{X}^{k}_t)$, $t \geq 0$
by letting each $\widehat{X}^{i}_t$ evolve starting from $z_i$ and following the arrows of $\mathscr{G}^{(i)}_\Lambda$; since the latter are independent and each single CCRW is distributed as a simple random walk, $\widehat{\mathbf{X}}_t$ is a system of $k$ independent continuous-time random walks.

Next, to define $\widehat{\mathbf{Y}}_t = (\widehat{Y}^{1}_t, \ldots, \widehat{Y}^{k}_t)$ 
(started at $\mathbf{z}$), we proceed as follows.
For a vector $\mathbf{y}=(y_1, \ldots, y_k) \in V^k$, 
denote by $S(\mathbf{y})= \{y_1, \ldots, y_k\}$ the set of coordinates in $\mathbf{y}$.
Now, for each $v \in V$, whenever 
$N_v \cap S(\widehat{\mathbf{Y}}_{t-}) = \cup_{j \in J}\{Y^{j}_{t-}\}$ 
and there are arrows of $\mathscr{G}^{(M_J)}_\Lambda$ (where $M_J = \max_{j \in J} j)$ 
at time $t$ from $N_v$ to $v$,
all $Y^j$ with $j \in J$ follow these arrows, i.e., $Y^{j}_t = v$ if there is an arrow from $Y^{j}_{t-}$ to $v$. Clearly, $\widehat{\mathbf{Y}}_t$ defined in this way
is a system of $k$ $\Lambda$-CCRWs started from $\mathbf{z}$.
Furthermore, defining, for $S \subset V$,
\begin{equation}\label{def:separation}
\Sep(S) := \inf \left\{ \dist(x,y) \colon x,y \in S, x \neq y\right\}
\end{equation}
(where $\inf \emptyset = \infty$), let
\begin{equation}\label{def:taubreak}
\taubreak := \inf \left\{t \geq 0 \colon\, \Sep(\widehat{\mathbf{Y}}_t) \leq 2 \right\}.
\end{equation}
A consequence of our construction is that
$\widehat{\mathbf{Y}}_t = \widehat{\mathbf{X}}_t$ for all finite $t \leq \taubreak$,
and so
\[
\taubreak = \inf \left\{t \geq 0 \colon\, \Sep(\widehat{\mathbf{X}}_t) \leq 2 \right\}
\]
as well. We also define
\begin{equation}\label{def:gbreak}
g^{(k)}_{\mathrm{break}}(\mathbf{z}) := \widehat{\Prob}(\taubreak < \infty).
\end{equation}
If~$\mathbf z$ has repeated entries, we extend this definition by removing the repetitions and applying~$g_\mathrm{break}$ to the resulting vector of distinct sites. In particular, $g^{(1)}_\mathrm{break} \equiv 0$, since the coupling never breaks in this case. 
The following lemma is the key property of this coupling.
\begin{lemma}\label{lemma:lim_g_Y_t}
Let~$k \in \N$, and let~$(\mathbf{Y}_t)_t$ be a~$k$-particle~$\Lambda$-CCRW. Then, for every~$\mathbf x \in V^k$,
$$\lim_{t \to \infty} E_\mathbf{x} g^{(k)}_\mathrm{break}( \mathbf Y_t )= 0.$$
\end{lemma}
\begin{proof}
Let~$\mathbf x = (x_1,\dots,x_k)$. 
We may assume that $k \geq 2$. By projective consistency, 
if the coupling of the~\(k\)-particle system breaks, then the coupling must break for at least one pair of particles.
Hence, by a union bound,
\begin{equation}\label{eq_E_Y_t_limit_sum}
E_\mathbf{x} g^{(k)}_\mathrm{break}(\mathbf Y_t) \leq \sum_{1 \leq i < j \leq k} E_{(x_i,x_j)} g^{(2)}_\mathrm{break}(Y^i_t,Y^j_t).
\end{equation}
It is therefore enough to consider two particles.
Let
\begin{equation}\label{def_tau_coal}
\tau_{\mathrm{coal}} := \inf\{t\ge0:Y_t^1=Y_t^2\}.
\end{equation}
If the particles have already coalesced by time~\(t\), then, for all $s \geq t$,
$$g_\mathrm{break}^{(2)}( Y^1_s,Y^2_s ) = g^{(1)}(Y^1_s) = 0. $$
On the other hand, if the coupling started from~\((Y_t^1,Y_t^2)\) subsequently breaks, the two particles must at some later time be at graph distance at most~\(2\). Consequently, by the Markov property,
for each $x \neq y$,
$$ E_{(x,y)} g^{(2)}_\mathrm{break}(Y^1_t,Y^2_t) \leq P_{(x,y)} \big( \exists s \ge t: d_G(Y^1_s,Y^2_s) \leq 2, \; t  < \tau_\mathrm{coal} \big).$$
As~$t \to \infty$, the right-hand side decreases to
\begin{equation}\label{coal_dog_prob}
P_{(x,y)} \big( \sup \{s > 0 \colon\, d_G(Y^1_s,Y^2_s) \leq 2\} = \infty, \;  \tau_\mathrm{coal} = \infty \big).
\end{equation}
We claim that this probability is zero. Indeed, bounded degrees and the form of the transition rates imply that
\begin{equation}\label{away_from_0_bd}
    \inf_{ \substack{u \neq v \\ d_G(u,v) \leq 2} } P_{(u,v)} \left(  \begin{array}{c}
        (Y^1_t)_t \text{ and } (Y^2_t)_t \text{ coalesce before they}  \\
        \text{are for first time at graph distance~3 }
    \end{array}  \right) =: \epsilon(\Lambda) > 0,
\end{equation}

To see this, if~\(u\) and~\(v\) are adjacent, an ordinary coalescing jump occurs with a rate bounded away from zero. If~\(d_G(u,v)=2\), the two vertices have a common neighbor. They may then coalesce through a coordinated jump whenever~\(\Lambda\) charges~\((0,1]\); in the remaining case the dynamics reduce to ordinary coalescing random walks, for which coalescence can occur through a fixed finite sequence of jumps. Since the degrees are uniformly bounded, the competing rates are uniformly bounded above, giving the claimed positive lower bound.

Then, by the strong Markov property and an iterative standard argument, we have that 
$$ \lim_{t \to \infty} E_{(x,y)} g^{(2)}_\mathrm{break}(Y^1_t,Y^2_t  ) \leq \lim_{n \to \infty} (1-\epsilon(\Lambda))^n = 0, \quad \text{for each } x \neq y. $$
This proves the result for two particles and hence, by~\eqref{eq_E_Y_t_limit_sum}, for every \(k\).
\end{proof}

The coupling construction also yields the following inequality involving~$g_\mathrm{break}$, which provides a useful tool for comparing~$(\widehat{\mathbf X}_t)_t$ and~$(\widehat{\mathbf Y}_t)_t$. In the proof of duality, Lemma~\ref{lemma:duality}, we introduced the notation~$P_\mathbf{y}$ for~$(\mathbf Y_t)_t$, we use the same notation to refer to the independent random walks~$( \mathbf X_t)_t$, since no ambiguity should arise. For every function~$f:V^k \to [0,1]$, vector~$\mathbf x \in V^k$ and time~$t \ge 0$,
\begin{equation}\label{ineq_X_Y_bound_g}
\tfrac{1}{2} \cdot | E_\mathbf{x}f( \mathbf X_t) - E_\mathbf{x}f( \mathbf Y_t) | \leq g^{(k)}_\mathrm{break}(\mathbf x).
\end{equation}

We now use the coupling to construct the measures~$\mu_{\alpha,\Lambda}$. For each probability measure~$\kappa$ on~$\{0,1\}^V$, we define
$$ \widehat{\kappa}(\mathbf x) := \kappa( \{\eta: \eta(x_1) = \cdots = \eta(x_k) = 1 \} ), \quad k \in \N \; \text{ and } \; \mathbf x = (x_1,\dots,x_k) \in V^k. $$

\begin{proof}[Proof of Theorem~\ref{thm:classification_general}, part 1]
Fix~$\alpha \in \mathscr H$ and consider~$\nu_\alpha$ from~\eqref{def_nu_alpha}. By compactness of~$\{0,1\}^V$, it is enough to prove convergence of the finite-dimensional distributions of~$\nu_\alpha S^\Lambda_t$ and by inclusion-exclusion, it suffices to show that for each~$k \in \N$ and~$\mathbf x \in V^k$ the limit
\begin{equation}\label{limit_exitence_part_1}
\lim_{t \to \infty}\widehat{\nu_\alpha S_t^\Lambda}(\mathbf x) \qquad \text{exists.}
\end{equation}
Then, since~$\mu_{\alpha,\Lambda}$ is obtained as a distributional limit of~$(\eta_t)_t$, it is stationary for the broadcast voter model dynamics~\cite[Proposition 1.8(d)]{Liggett85}.

We now prove~\eqref{limit_exitence_part_1} for fixed~$k \in \N$ and~$\mathbf x \in V^k$. By duality~\eqref{duality_equality_raw}, 
$$ \widehat{ \nu_\alpha S^\Lambda_t }(\mathbf z) = E_\mathbf{z}\widehat{\nu}_\alpha(\mathbf Y_t) \quad \text{for all } \mathbf z \in V^k \; \text{ and } \; t\ge 0.$$
By harmonicity of~$\alpha$ and independence of the coordinates of~$(\mathbf X_t)_t$,
$$ \widehat{\nu}_\alpha(\mathbf z) = E_\mathbf{z} \widehat{\nu}_\alpha(\mathbf X_t) \quad \text{ for all } \mathbf z \in V^k \; \text{ and } \; t\ge 0, $$
where for the above we recall that,
if~\(\mathbf z\) has repeated entries, we use the convention above of removing repetitions before starting the independent system.

Then, by the Markov property, for any~$s,t \ge 0$,
\begin{equation}\label{limit_pre_part_1}
\left| \widehat{ \nu_\alpha S_{t+s}^\Lambda }(\mathbf{ x}) - \widehat{ \nu_\alpha S^\Lambda_{t} }(\mathbf{ x}) \right| \leq E_\mathbf{x} \left| E_{\mathbf Y_t}[ \widehat{\nu}_\alpha(\mathbf Y_s) ] - E_{\mathbf Y_t}[ \widehat{\nu}_\alpha(\mathbf X_s) ] \right| \stackrel{\eqref{ineq_X_Y_bound_g}} \leq 2 \cdot E_\mathbf{x} g^{(k)}_\mathrm{break}(\mathbf Y_t).
\end{equation}
Applying Lemma~\ref{lemma:lim_g_Y_t}, we obtain the existence of the limit in~\eqref{limit_exitence_part_1}, finishing the proof.
\end{proof}

\subsection{The infinite-collision regime}\label{ss:part2_thm_sm}
We now prove the second assertion of Theorem~\ref{thm:classification_general}. Throughout this subsection, we assume that~$G$ has the infinite collision property. The main step is to show that the infinite collision property of the underlying independent random walks forces two particles of the~\(\Lambda\)-CCRW to coalesce almost surely. We then combine this with duality to identify all stationary measures. 

Let \((X_t^1,X_t^2)_{t\ge0}\) be two independent random walks with generator \(\mathcal G_{\Lambda,\mathrm{irw}}^{(2)}\), from~\eqref{matrix_q_irw}. Recall that the infinite collision property is defined for two walks started from a common vertex. In continuous time, however, the same property holds for arbitrary starting positions. Indeed, the pair process is irreducible on~\(V^2\), and the strong Markov property applied at the first hitting time of any~\((x,y)\in V^2\) 
shows that infinite collisions from a common starting vertex imply infinite collisions from~\((x,y)\). Consequently, for every~$x,y \in V$,
$$ P_{(x,y)} \left( X_t^1=X_t^2 \text{ for some }t\ge0 \right) = 1. $$
We next establish the analogous coalescence property for the~\(\Lambda\)-CCRW.
\begin{lemma}\label{lem:coup_to_0}
Let~\((Y_t^1,Y_t^2)_{t\ge0}\) be a two-particle~\(\Lambda\)-CCRW. Then, for every  \(x,y\in V\),
$$ P_{(x,y)}( \exists t \ge 0: Y^1_t = Y^2_t ) = 1.   $$
\end{lemma}
\begin{proof}
Recall the coupling introduced in the previous section and the stopping time~$\tau_\mathrm{coal}$ and the positive constant~$\epsilon(\Lambda)$ from~\eqref{def_tau_coal} and~\eqref{away_from_0_bd}, respectively. Since the independent pair collides almost surely, the coupling time~\(\taubreak\) is finite almost surely. By the construction of the coupling, at this time either the two particles have already coalesced, or they are at graph distance at most~\(2\).
If coalescence does not occur during this first attempt, wait
until the particles reach graph distance~\(3\) and restart the coupling from their new positions. The infinite collision property of the independent walks again implies that the new coupling breaks almost surely, and hence either coalescence occurs or the particles return to graph distance at most \(2\). Repeating this construction and applying the strong Markov property at the successive attempts gives
$$ P_{(x,y)}( \tau_\mathrm{coal} = \infty ) \leq (1-\epsilon(\Lambda) )^n, \qquad \text{for every } n \in \N. $$
Letting~$n \to \infty$ yields~$P_{(x,y)}( \tau_\mathrm{coal} = \infty ) = 0$, which proves the claim.
\end{proof}

We are now ready to prove the second assertion of Theorem~\ref{thm:classification_general}.

\begin{proof}[Proof of Theorem~\ref{thm:classification_general}, part 2.]
Let~$\mu \in \mathscr I_\Lambda$. By stationarity and duality~\eqref{duality_equality_raw}, 
for every~$x,y \in V$ and~$t \ge 0$,
\begin{equation*}
\mu(\eta: \eta(x) \neq \eta(y)) = \int P_\eta( \eta_t(x) \neq \eta_t(y) ) \; \mu(\mathrm{d}\eta) = \int P_{(x,y)}( \eta(Y^1_t) \neq \eta(Y^2_t) ) \; \mu(\mathrm{d}\eta).
\end{equation*}
The probability inside the integral is bounded above by~$P_{(x,y)}(\tau_\mathrm{coal}>t)$. Taking the limit as~$t \to \infty$ this vanishes by Lemma~\ref{lem:coup_to_0}. Hence,~$\mu$ is supported on the consensus configurations. Consequently, there exists~$\alpha \in [0,1]$ such that~$\mu = \alpha \delta_{\overline{1}} + (1-\alpha)\delta_{\overline{0}}.$ Then,
$$ \mathscr I_\Lambda = \{ \alpha \delta_{\overline{1}} + (1-\alpha)\delta_{\overline{0}}: \alpha \in [0,1]  \} $$
and therefore~$\mathscr I_{\Lambda,e} = \{\delta_{\overline{0}},\delta_{\overline{1}}\}$. This proves the second assertion of Theorem~\ref{thm:classification_general}.
\end{proof}

\subsection{The finite-collision regime}\label{ss:part3_thm_sm}
We now prove the third assertion of Theorem~\ref{thm:classification_general}. Throughout this subsection, we assume that~\(G\) has the finite collision property. The proof follows the general strategy developed for the classical voter model in Chapter~V of~\cite{Liggett85}, with the coupling introduced in Section~\ref{ss:part1_thm_sm} used to compare the~\(\Lambda\)-CCRW with independent random walks. We first establish several preliminary results and then use them to characterize the extremal stationary measures.

\subsubsection{Preliminaries}
We begin with a criterion for identifying the measures~\(\mu_{\alpha,\Lambda}\). 
Define
$$ \Xi_t^x(\eta) := E_x\big[\eta(X_t)\big], 
\qquad x\in V, \; \eta\in\{0,1\}^V,
$$
where \((X_t)_{t\ge0}\) is a random walk that jumps along every incident edge at rate~$\Lambda[0,1]$.

\begin{proposition}\label{prop:main_part3}
Let~$\mu \in \mathscr I_\Lambda$ and~$\alpha \in \mathscr H$. Then~$\mu = \mu_{\alpha,\Lambda}$ if and only if, for every~$x \in V$,
\begin{equation}\label{Xi_def_t}
\Xi^x_t(\eta):= E_x\eta(X_t) \xlongrightarrow[t \to \infty]{L^2(\mu)} \alpha(x).
\end{equation}
\end{proposition}
Since~$0 \leq \Xi^x_t \leq 1$, the~$L^2(\mu)$-convergence above may equivalently be replaced by convergence in~$\mu$-probability or by convergence in~$L^p(\mu)$ for any~$p \ge 1$.

The proof of Proposition~\ref{prop:main_part3} will use the following consequence of the finite collision property. It is the counterpart, for the independent system, of Lemma~\ref{lemma:lim_g_Y_t}.

\begin{lemma}\label{lemma:main_part3}
Let~$k \in \N$ and let~$(\mathbf X_t)_t$ be the system of~$k$ independent random walks with generator~$\mathcal G^{(k)}_{\Lambda,\mathrm{irw}}$ from~\eqref{matrix_q_irw}. Then, for every~$\mathbf x \in V^k$,
$$ \lim_{t \to \infty} E_\mathbf{x} g^{(k)}_\mathrm{break}(\mathbf X_t) = 0. $$
\end{lemma}
\begin{proof}
Let~$\mathbf x = (x_1,\dots,x_k)$. 
We may assume that $k \geq 2$.
By the construction of the coupling in Section~\ref{ss:part1_thm_sm}, if the coupling breaks, then at some time at least two particles of the independent system must be at graph distance at most~2. Hence,
\begin{equation}\label{ineq_coup_irw_only}
\widehat{\Prob}( \taubreak <\infty ) \leq \widehat{\Prob}( \exists \; 1\leq i < j \leq k \text{ and } s \ge 0: d_G( \widehat{X}^i_s,\widehat{X}^j_s ) \leq 2 ).
\end{equation}
Applying this bound at time \(t\) and using the Markov property and a union bound, we obtain
$$  E_\mathbf{x} g^{(k)}_\mathrm{break}(\mathbf X_t) \leq \sum_{1 \leq i < j \leq k} P_{(x_i,x_j)} \big( \exists s \ge t: d_G(X^i_s,X^j_s) \leq 2 \big) $$
It therefore suffices to prove that, for every \(x,y\in V\),
\begin{equation*}
\lim_{t \to \infty} P_{(x,y)}( \exists s \ge t: d_G(X^1_s,X^2_s) \leq 2 ) = 
P_{(x,y)} \left( \sup\{ s > 0 \colon\, d_G(X_s^1,X_s^2)\le2 \} = \infty \right) = 0.
\end{equation*}
To prove this, we use a similar argument as in Lemma~\ref{lemma:lim_g_Y_t}. 
Clearly, if the two walks are within distance two, in finite time they must reach distance at least three; thus the only way to make the set of times at which they are within distance $2$ unbounded is to successively reach distance $\geq 3$ and then $\leq 2$ back and forth, infinitely often.
In each such loop, after reaching distance $\leq 2$, 
the two random walks meet with a probability bounded from below by a certain $\epsilon>0$ before reaching distance $\geq 3$. The strong Markov property applied at such successive times would give infinitely many collision attempts, each with probability at least~$\epsilon$ of succeeding. The Borel-Cantelli lemma then implies that infinitely many of these attempts succeed almost surely, contradicting the finite collision property. The result then follows from \eqref{ineq_coup_irw_only}.
\end{proof}

We now prove the identification criterion.
\begin{proof}[Proof of Proposition~\ref{prop:main_part3}]
Let~$\mu \in \mathscr I$ and~$\alpha \in \mathscr H$. Assume~\eqref{Xi_def_t} holds. Then, for any~$k \in \N$ and~$\mathbf x = (x_1,\dots,x_k)$ with distinct entries,
\begin{equation}\label{proof_prop_1}
\lim_{t \to \infty} E_\mathbf{x}\widehat{\mu}(X^1_t,\dots,X^k_t) = \lim_{t \to \infty} \int \Xi^{x_1}_t(\eta) \cdots \Xi^{x_k}_t(\eta) \; \mu(\mathrm{d}\eta) = \alpha(x_1) \cdots \alpha(x_k) = \widehat{\nu}_\alpha(\mathbf x).
\end{equation}
By stationarity and the Markov property we have for any~$s,t \ge 0$,
\begin{equation*}
    |\widehat{\mu}(\mathbf x) - E_\mathbf{x} E_{\mathbf Y_t} \widehat{\mu}(\mathbf X_s)  | = |E_\mathbf{x} E_{\mathbf Y_t} \widehat{\mu}(\mathbf Y_s) - E_\mathbf{x} E_{\mathbf Y_t} \widehat{\mu}(\mathbf X_s) | \stackrel{\eqref{ineq_X_Y_bound_g}} \leq 2 \cdot E_\mathbf{x}[g^{(k)}_\mathrm{break}(\mathbf Y_t)].
    \end{equation*}
Taking the limit as~$s \to \infty$ an using the bounded convergence theorem with~\eqref{proof_prop_1}, we have
\begin{equation}\label{ineq_mu_nu_alpha_g}
|\widehat{\mu}(\mathbf x) - E_\mathbf{x}[\widehat{\nu}_\alpha(\mathbf Y_t)]| \le 2 \cdot E_\mathbf{x}[g^{(k)}_\mathrm{break}(\mathbf Y_t)].
\end{equation}
Using Lemma~\ref{lemma:lim_g_Y_t} together with the definition of~$\mu_\alpha$, we obtain~$\mu = \mu_\alpha$ after taking the limit as~$t \to \infty$.

Conversely, assume that~$\mu = \mu_{\alpha,\Lambda}$. Fix~$x \in V$, we want to prove that~$\Xi^x_t$ converges in~$L^2(\mu_{\alpha,\Lambda})$ to~$\alpha(x)$. Note that
\begin{equation}\label{e:reduc_cond_conv}
E_{\mu_{\alpha,\Lambda}}( \Xi^x_t (\eta)- \alpha(x) )^2 = E_{(x,x)}\widehat{\mu}_{\alpha,\Lambda}(X^1_t,X^2_t) - 2\alpha(x)E_x \widehat{\mu}_{\alpha,\Lambda}(X^1_t) + \alpha(x)^2.
\end{equation}
By harmonicity of~$\alpha$, to show that the previous expression goes to~0 as~$t \to \infty$, it suffices to prove that~$E_{(x,x)}\widehat{\mu}_{\alpha,\Lambda}(X^1_t,X^2_t)$ converges to~$\alpha(x)^2$ as~$t \to \infty$. 

By harmonicity and the Markov property, for any~$s,t \ge 0$,
\begin{align*}
| \alpha(x)^2 - E_{(x,x)} E_{\mathbf X_t} \widehat{\nu}_\alpha(\mathbf Y_s) | & = | E_{(x,x)} E_{\mathbf X_t} \widehat{\nu}_\alpha(\mathbf X_s) - E_{(x,x)} E_{\mathbf X_t} \widehat{\nu}_\alpha(\mathbf Y_s) | \\[2mm]
& \hspace{-1.2mm} \stackrel{\eqref{ineq_X_Y_bound_g}} \leq 2 \cdot E_{(x,x)} \big[ g^{(2)}_\mathrm{break}( \mathbf X_t ); X^1_t \neq X^2_t \big] + P_{(x,x)}(X^1_t = X^2_t).
\end{align*}
Analogously to the argument above, we first let~$s \to \infty$ to obtain~$\mu_{\alpha,\Lambda}$. Then, as~$t \to \infty$, the first term tends to zero by Lemma~\ref{lemma:main_part3} and the second by the finite collision property. This completes the proof.
\end{proof}

The final preliminary result concerns subsequential limits of the dynamics. Define
\[
\mathscr{I}_{\Lambda, \infty} := \left\{\begin{array}{c}\mu:  \text{there exists a probability measure $\nu$ on $\{0,1\}^V$ and a sequence $(t_n)_{n \ge 0}$}\\
\text{with $t_n \to \infty$ such that $\nu S_{t_n}$ converges weakly to $\mu$ as $n \to \infty$}\end{array}\right\}.
\]
The following result is the analogue, in the present setting, of Corollary~1.28 in Chapter~V of~\cite{Liggett85}. 
\begin{lemma}\label{lemma_I_infty}
$\mathscr I_\Lambda = \mathscr I_{\Lambda,\infty}$.
\end{lemma}

The inclusion~$\mathscr I_\Lambda \subseteq \mathscr{I}_{\Lambda,\infty}$ is immediate. The reverse inclusion is not automatic from subsequential convergence: in general, a subsequential limit of a Markov process need not be stationary. By contrast, if~\(\nu S_t^\Lambda\) converges as~\(t\to\infty\), then its limit is stationary by Proposition~1.8 of~\cite{Liggett85}. In the present setting, the reverse inclusion 
will be a consequence of the duality established in Lemma~\ref{lemma:duality}.
\begin{proof}
For the reverse inclusion, we apply Corollary 3.2 of~\cite{AstoquillcaValesin26}. For every~\(k\in\mathbb N\), the dual state space~\(V^k\) is countable. If~\(c_k(\mathbf x)\) denotes the total jump rate of the~\(k\)-particle~\(\Lambda\)-CCRW, then Lemma~\ref{lemma:CCRW_properties} and the one-particle marginal property give
$$ \Lambda[0,1] \leq c_k(\mathbf x) \le k\Lambda[0,1] \cdot \sup_{x\in V} \deg(x), \qquad \mathbf x\in V^k. $$
Hence the rate conditions in (3.1) of the cited result are satisfied. Moreover,~\(D^{(k)}(\eta,\mathbf y)=\prod_i\eta(y_i)\) is bounded and continuous, condition~(ii) holds with~\(M=1\), and the family~\(\{D^{(k)}\}\) is measure-determining by inclusion-exclusion. 
This concludes the proof.
\end{proof}

Now we are ready to give the:

\begin{proof}[Proof of Theorem~\ref{thm:classification_general}, part 3.]
Our goal is to prove that, in this case,
$$ \mathscr I_{\Lambda,e} = \{ \mu_{\alpha,\Lambda}: \alpha \in \mathscr H \}. $$
We first show that every measure in the family on the right-hand side is extremal, and then prove the reverse inclusion.

\paragraph{Extremality of~$\mu_{\alpha,\Lambda}$.} Fix~$\alpha \in \mathscr H$, and suppose that, for some~$\theta \in (0,1)$ 
$$ \mu_{\alpha,\Lambda} = \theta \cdot \mu_1 + (1-\theta) \cdot \mu_2, \qquad \mu_1,\mu_2 \in \mathscr I_\Lambda. $$
By Proposition~\ref{prop:main_part3}, for each~$x \in V$, 
$$ \int (\Xi^x_t(\eta) - \alpha(x))^2 \mu_{\alpha,\Lambda}(\mathrm{d} \eta) \to 0 \quad \text{as} \quad t \to \infty. $$
On the other hand,
$$ \int (\Xi^x_t(\eta) - \alpha(x))^2 \mu_{\alpha,\Lambda}(\mathrm{d} \eta) = \theta \cdot \int (\Xi^x_t(\eta) - \alpha(x))^2 \mu_1(\mathrm{d} \eta) + (1-\theta) \cdot \int (\Xi^x_t(\eta) - \alpha(x))^2 \mu_2(\mathrm{d} \eta) $$
Since both terms on the right-hand side are nonnegative, each of them must converge to zero. Again, by Proposition~\ref{prop:main_part3}, 
 $\mu_1$ and $\mu_2$
 must coincide with~$\mu_{\alpha,\Lambda}$. Hence,~$\mu_{\alpha,\Lambda}$ is extremal, and therefore
$$ \{ \mu_{\alpha,\Lambda}: \alpha \in \mathscr H \} \subseteq \mathscr I_{\Lambda,e}.$$

\paragraph{The reverse inclusion.} We begin by proving that for each measure~$\mu \in \mathscr I_\Lambda$, integer~$k \in \N$ and vector~$\mathbf x \in V^k$ the limit
$$ \lim_{t \to \infty} E_\mathbf{x}\widehat{\mu}(\mathbf X_t) \quad \text{exists.} $$
Indeed, for every~$s,t \ge 0$, stationarity, duality and the Markov property give
$$ |E_\mathbf{x} \widehat{\mu}(\mathbf X_{t+s}) - E_\mathbf{x} \widehat{\mu}(\mathbf X_{t}) | \stackrel{\eqref{duality_equality_raw}} =  |E_\mathbf{x} E_{\mathbf X_t} \widehat{\mu}( \mathbf X_s) - E_\mathbf{x} E_{\mathbf X_t} \widehat{\mu}(\mathbf Y_s) | \stackrel{\eqref{ineq_X_Y_bound_g}} \leq 2 \cdot E_\mathbf{x} g_\mathrm{break}( \mathbf X_t ).$$
Then, due to Lemma~\ref{lemma:main_part3}, we have the desired convergence. We define for each~$k \in \N$,
\begin{equation}\label{h_function_part3}
h^{(k)}(\mathbf x) := \lim_{t \to \infty} E_\mathbf{x}\widehat{\mu}(\mathbf X_t), \quad  \mathbf x \in V^k.
\end{equation}
We use this function to establish the inclusion. By the Markov property and bounded convergence,
$$ E_\mathbf{x}h^{(k)}(\mathbf X_s) = h^{(k)}(\mathbf x), \qquad s\ge0, $$
so~$h^{(k)}$ is harmonic for the~$k$-system of independent random walks. 

We divide the proof into two cases according to whether the graph is recurrent or transient.  The recurrent case is an adaptation of Theorem~2.2 (B) in~\cite{Astoquillca26}.

\paragraph{The recurrent case.} Assume~$G$ is recurrent. 
By Orey's theorem, the tail~\(\sigma\)-algebra of the one-particle (discrete-time) random walk is trivial. Using independence, the corresponding product argument for finitely many independent copies is the same as in the proof of Theorem 2.2(B) of~\cite{Astoquillca26}; the continuous-time version follows as in the proof of Proposition 2.1 of \cite{Asto2026Coll}. It follows that every bounded harmonic function for the \(k\)-particle independent system is constant. Hence
$$ h^{(k)}(\mathbf x) = h_k, \qquad \text{for every } \mathbf x \in V^k,  $$
for some number~$h_k \in [0,1]$ depending only on~$k$.

Fix~$\mu \in \mathscr I_\Lambda$. For~\(t\ge0\), we introduce an auxiliary probability measure \(\mu_t\) on~$\{0,1\}^V$ as follows. First sample~$\eta$ according to~$\mu$, and, conditionally on~$\eta$, sample a configuration~$\xi$ according to 
$$ \bigotimes_{x \in V} \mathrm{Ber}( \Xi^x_t(\eta) ). $$
We denote by~$\mu_t$ the unconditional law of~$\xi$. Note that for every vector~\(\mathbf x=(x_1,\ldots,x_k)\) with distinct entries,
$$ \widehat\mu_t(\mathbf x) = \int \Xi^{x_1}_t(\eta) \cdots \Xi^{x_k}_t(\eta) \; \mu(\mathrm d\eta) = E_{\mathbf x}\widehat\mu(\mathbf X_t). $$
Then,~$\widehat\mu_t(\mathbf x)\longrightarrow h_k$ as~\(t\to\infty\). Hence~\(\mu_t\) converges weakly to a probability measure~\(\mu_\infty\) whose~\(k\)-point functions depend only on~\(k\). Thus~\(\mu_\infty\) is exchangeable.

By de Finetti’s theorem, there exists a (uniquely determined) probability measure~$\gamma$ on the Borel sets of~$[0,1]$ such that
$$ \mu_\infty = \int_{[0,1]} \nu_\alpha \; \gamma(\mathrm{d}\alpha). $$

Noting that, by duality and Markovianity, 
$\widehat{\mu}(\mathbf{x}) = \widehat{\mu S^{\Lambda}_{t+s}}(\mathbf{x}) = E_{\mathbf{x}} E_{\mathbf{Y}_t} \widehat{\mu}(\mathbf{Y}_s)$ 
and that
\[
E_{\mathbf{x}} E_{\mathbf{Y}_t} \widehat{\mu}(\mathbf{X}_s) = E_{\mathbf{x}} \widehat{\mu_s}(Y_t) 
\underset{s \to \infty}{\longrightarrow} E_{\mathbf{x}} \widehat{\mu_\infty}(\mathbf{Y}_t) = \int_{[0,1]} \widehat{\nu_\alpha S^\Lambda_t}(\mathbf x) \; \gamma(\mathrm{d}\alpha),
\]
we argue as for~\eqref{limit_pre_part_1} to obtain

$$ \left| \widehat{\mu}(\mathbf x) - \int_{[0,1]} \widehat{\nu_\alpha S^\Lambda_t}(\mathbf x) \; \gamma(\mathrm{d}\alpha) \right| \leq 2 \cdot E_\mathbf{x} g^{(k)}_\mathrm{break}(\mathbf Y_t) \quad \text{for each } k \in \N \text{ and } \mathbf x \in V^k. $$

Taking taking~$t \to \infty$ and applying Lemma~\ref{lemma:lim_g_Y_t} we obtain $\mu = \int^1_0 \mu_{\alpha, \Lambda} \; \gamma(\mathrm{d}\alpha)$. 

Suppose now that~$\mu \in \mathscr I_{\Lambda,e}$. We claim that~$\gamma$ must be a Dirac mass. Indeed, otherwise there exists~$c \in (0,1)$ such that~$0 < \gamma[0,c] < 1$. Define
$$ \mu_1 = \frac{1}{\gamma[0,c]} \cdot \int_{[0,c]} \mu_{\alpha,\Lambda} \; \gamma(\mathrm d \alpha), \qquad \mu_2 = \frac{1}{1- \gamma[0,c]} \cdot \int_{(c,1]} \mu_{\alpha,\Lambda} \; \gamma(\mathrm d \alpha) $$
Since every \(\mu_{\alpha,\Lambda}\) is stationary, both~\(\mu_1\) and~\(\mu_2\) belong to~\(\mathscr I_\Lambda\), and
$$ \mu = \theta \cdot \mu_1 + (1-\theta) \cdot \mu_2, \quad \text{with } \; \theta = \gamma[0,c]. $$
Furthermore, their one-site marginals are different:
$$ \widehat{\mu}_1(x) = \frac{1}{\theta} \int_{[0,c]} \alpha \; \gamma(\mathrm d \alpha) \leq c < \frac{1}{1-\theta} \int_{ (c,1] }  \alpha \; \gamma(\mathrm d \alpha) = \widehat{\mu}_2(x).$$
Thus,~$\mu_1 \neq \mu_2$, contradicting the extremality of~$\mu$. Hence,~$\gamma = \delta_a$ for some~$a \in [0,1]$, and consequentely~$\mu = \mu_{a,\Lambda}$. Since every bounded harmonic function on a recurrent connected graph is constant, this proves the desired reverse inclusion in the recurrent case.

\paragraph{The transient case.} Assume~$G$ is transient. Let~$\mu \in \mathscr I_{\Lambda,e}$, and define~$\alpha(x):=\widehat{\mu}(x)$,~$x \in V$. The stationarity of~$\mu$ and the one-particle duality imply that~$\alpha$ is harmonic with respect to the simple random walk,
i.e., $\alpha \in \mathscr H$. Let us prove that~$\mu = \mu_{\alpha,\Lambda}$. By Proposition~\ref{prop:main_part3} and \eqref{e:reduc_cond_conv}, it is enough to show that
\begin{equation}\label{eq_lim_h_2}
h^{(2)}(x,x) \stackrel{\eqref{h_function_part3}}= \lim_{t \to \infty} E_{(x,x)}\widehat{\mu}(X^1_t,X^2_t) = \alpha(x)^2 \quad \text{for each } x \in V.
\end{equation}
We use three ingredients. 
First, extremality of~$\mu$, together with Lemma~\ref{lemma_I_infty}, allows us to apply the argument of Lemma~1.30 in Chapter~V of~\cite{Liggett85}. It yields the following asymptotic decorrelation:
\begin{equation}\label{part3_final_1}
\text{If } \mu \in \mathscr I_{\Lambda,e} \quad \text{ then } \quad \lim_{t \to \infty} E_x\widehat{\mu}(X^1_t,y) = \widehat{\mu}(x) \cdot \widehat{\mu}(y) \quad \text{for all } x,y \in V.
\end{equation}
Second, following the proof of Lemma 1.32 in Chapter V of~\cite{Liggett85}, Corollary~7.3 of Chapter~II implies that \(h^{(2)}\) is harmonic in each coordinate separately. To apply the discrete-time statement, we use the standard uniformization of the continuous-time walk: one adds suitable holding probabilities so that the chain is represented by a discrete-time Markov kernel with a constant clock rate. For the two-particle system, this kernel is the average of the two commuting coordinate kernels, so the cited corollary applies. Hence
\begin{equation}\label{part3_final_2}
E_x h^{(2)}(X^1_t,y) = h^{(2)}(x,y) \quad \text{for all } x,y \in V \text{ and } t \ge 0.
\end{equation}
Third, by the coupling inequality,
$$ \left| E_x h^{(2)}(X^1_t,x) - E_x\widehat{\mu}(X^1_t,x) \right| \stackrel{\eqref{ineq_X_Y_bound_g}} \leq 2 \cdot E_x g^{(2)}_\mathrm{break}(X^1_t,x) + P_x(X_t=x) \quad \text{ for all } x\in V. $$
Taking the limit~$t \to \infty$ on both sides and using transience together with~\eqref{part3_final_1} and~\eqref{part3_final_2}, we conclude that to have~\eqref{eq_lim_h_2} it suffices to prove that
\begin{equation}\label{eq_lim_g_coup_X^1=0}
\lim_{t \to \infty} E_x g^{(2)}_\mathrm{break}(X^1_t,x) = 0 \quad \text{for all } x \in V.
\end{equation}

From the coupling construction, we have, for each~$x \neq y$,
\begin{equation}\label{ineq_g_int_d_G}
g^{(2)}_\mathrm{break}(x,y) \leq P_{x,y} \left( \exists s \ge 0: d_G (X^1_s,X^2_s) \leq 2 \right).
\end{equation}
To further upper bound this, we define~$\tau := \inf\{t \ge 0: d_G(X^1_s,X^2_s) \leq 2 \}$. Then, by the strong Markov property,
$$ E_{x,y}\left[ \left.\int^\infty_0 \mathds{1}\{ d_G(X^1_s,X^2_s) \leq 2 \} \; \mathrm{d}s \right| \mathcal F_\tau \right] \ge \mathds{1}\{ \tau < \infty \} \cdot P_{(X^1_\tau,X^2_\tau)} \left( \begin{array}{c}
    \text{No jump of } (X^1_s,X^2_s) \\
    \text{ during } [0,1] 
\end{array} \right).$$
By the bounded degree assumption, the last probability is bounded from below by~$C := \exp( - 2 \Lambda[0,1] \cdot \sup_{x \in V} \deg(x) )$. 
Taking expectations on both sides and using~\eqref{ineq_g_int_d_G} we get
\begin{equation}\label{ineq_q_int_d_G_1}
g^{(2)}_\mathrm{break}(x,y) \leq C^{-1} \cdot \int^\infty_0 P_{x,y} \left( d_G(X^1_s,X^2_s) \leq 2 \right)  \mathrm{d}s.
\end{equation}
Denote by $p_t$
the transition kernel of the random walk, i.e.,
$$ p_t(x,y) := P_x(X^1_t = y) \quad \text{for } t \ge 0 \text{ and } x,y \in V. $$
Since~\(G\) has bounded degrees and each edge is crossed at rate~\(\Lambda[0,1]>0\), there exists~$c>0$ such that
\begin{equation}\label{eq_c_p_1_M}
p_1(x,y) \ge c \quad \text{for all } x,y \text{ with } \; d_G(x,y) \leq 2.
\end{equation}
Fix~$x \in V$. Using independence and the Chapman-Kolmogorov identity, we have, by~\eqref{ineq_q_int_d_G_1},
\begin{align*}
E_xg_\mathrm{break}(X^1_t,x) & \leq C^{-1} \cdot \int^\infty_0 \sum_{z_1,z_2: d_G(z_1,z_2) \leq 2} p_{s+t}(x,z_1) p_s(x,z_2) \; \mathrm{d}s \\[2mm]
& = C^{-1} \cdot \int^\infty_0 \sum_{z_2} p_s(x,z_2) \sum_{z_1: d_G(z_1,z_2) \leq 2} p_{s+t}(x,z_1)  \; \mathrm{d}s \\[2mm]
& \leq (Cc)^{-1} \cdot \int^\infty_0 p_{2s+t+1}(x,x) \; \mathrm{d}s = \frac{(Cc)^{-1}}{2} \cdot \int^\infty_{t+1}p_r(x,x) \; \mathrm{d}r, 
\end{align*}
where the last inequality follows from Chapman-Kolmogorov, the lower bound in~\eqref{eq_c_p_1_M} and 
the fact that
$p_t(x,y) = p_t(y,x)$. Finally, by transience, the 
last integral above tends to zero as $t \to \infty$, 
verifying
\eqref{eq_lim_g_coup_X^1=0}. 
Thus
$\mu = \mu_{\alpha,\Lambda}$ and
we have the reverse inclusion.
\end{proof}

\subsection{Proof of Corollary~\ref{c:long_lattice}}
\label{ss:proof_corollary}

For every~\(R\in\mathbb N\), the graph~\(G_{d,R}\) is transitive and has bounded degrees, and its simple random walk is a symmetric irreducible finite-range random walk on~\(\mathbb Z^d\). Its recurrence classification is therefore the same as that of nearest-neighbor simple random walk: it is recurrent for~\(d=1,2\) and transient for~\(d\geq3\); see for instance, \cite[Theorem~4.1.1]{LawlerLimic10}.  
On transitive graphs, recurrence and transience correspond to the infinite and finite collision properties, respectively. Moreover, by the Choquet--Deny theorem \cite{ChoquetDeny60}, every bounded harmonic function for a symmetric irreducible finite-range random walk on \(\mathbb Z^d\) is constant.  
The first and second assertions therefore follows from Theorem~\ref{thm:classification_general}.

For the third assertion, we prove that~\(\mu_{\alpha,\Lambda}\) is mixing.  The argument is analogous to the proof of Proposition 1.2 in~\cite{Astoquillca26}, to which we refer for further details. Translation invariance follows from the translation invariance of both the dynamics and the initial measure~\(\nu_\alpha\). 

Let~$(\mathbf Y_t)_t$ be a~$k$-particle~$\Lambda$-CCRW. By the duality relationship~\eqref{duality_equality_raw}, 
$$ P_{\nu_\alpha}( \eta_t(x) = 1, \; \forall x \in S(\mathbf x) ) = E_\mathbf{x} \alpha^{| S( \mathbf Y_t  ) |}. $$
for every time~$t \ge 0$,~$k \in \N$ and~$\mathbf x \in V^k$. Since~$| S( \mathbf Y_t ) |$ is non-increasing in~$t$, we can define~$|S( \mathbf Y_\infty )| := \lim_{t \to \infty} | S( \mathbf Y_t ) |$. Taking~$t \to \infty$, we obtain
\begin{equation}\label{eq_moment_rep}
\mu_{\alpha,\Lambda}(\eta: \eta(x) = 1, \; \forall x \in S(\mathbf x) ) = E_\mathbf{x} \alpha^{| S( \mathbf Y_\infty) |}.
\end{equation}
Now let \(\mathbf x\) and \(\mathbf y\) be two finite vectors in~\(\Z^d\). For \(z\in\mathbb Z^d\), consider two independent~\(\Lambda\)-CCRWs,~\((\mathbf X_t)_t\) and~\((\mathbf Y_t)_t\), started from~\(\mathbf x\) and~\(\mathbf y+z\), respectively. Let
$$ \sigma_z = \inf\{t\ge0:\operatorname{dist}( S(\mathbf X_t),S(\mathbf Y_t))\le 2R\}, \qquad z \in \Z^d. $$
Until time~\(\sigma_z\), there is no vertex whose broadcast neighborhood contains particles from both systems. Consequently, the pair~$(\mathbf X_t, \mathbf Y_t)$ can be coupled with the~\(\Lambda\)-CCRW started from \( (\mathbf x, \mathbf y + z)\) so that they agree up to time~\(\sigma_z\).  From this point, the proof proceeds exactly as in the reference. By a union bound and the coupling with independent random walks,
$$ \mathbb P(\sigma_z<\infty) \leq \sum_{u\in S(\mathbf x)} \sum_{v\in S(\mathbf y)+z} P_{u,v} \left( \exists t\geq0: \|X_t-Y_t\|_1\leq2R \right), $$
where~\( (X_t)_t\) and~\( (Y_t)_t\) are independent copies of the underlying symmetric finite-range random walk. The estimate of Lemma~3.1 in~\cite{Astoquillca26} extends verbatim to symmetric finite-range random walks on~\(\mathbb Z^d\); hence the right-hand side tends to zero as~$\|z\|_1 \rightarrow\infty.$ Using the moment representation~\eqref{eq_moment_rep}, we conclude 
as in \cite{Astoquillca26}
that~\(\mu_{\alpha,\Lambda}\) is mixing (in particular ergodic) with respect to translations of \(\mathbb Z^d\).
\hfill \qed

\section{Dependence on the broadcasting measure}\label{ss_dependence_Lambda}
Sections~\ref{s:TheBVModel} and~\ref{s:stat_meas} reveal two different roles played by the measure~\(\Lambda\). First, at the level of a single site, the~\(\Lambda\)-broadcast voter model has the same flip rates as the standard voter model, up to the deterministic time scaling~\(\Lambda[0,1]\); equivalently, each coordinate of the dual process is a simple random walk run at this speed. Second, the dependence on~\(\Lambda\) appears only through the correlations created by simultaneous broadcasts, or, dually, through coordinated jumps of several particles. Having characterized the space of stationary measures with~$\mu_{\alpha,\Lambda}$ in Section~\ref{s:stat_meas}, we now ask to what extend these measures retain information about the broadcast mechanism~$\Lambda$.

There are two cases in which the stationary measure is immediately independent of the broadcast measure. If~$\alpha \equiv 0$ or~$\alpha \equiv 1$, then
$$ \mu_{\alpha,\Lambda} = \delta_{\overline{0}} \qquad \text{or} \qquad \mu_{\alpha,\Lambda} = \delta_{\overline{1}}, $$
respectively. Moreover, if~\(G\) has the infinite collision property, Theorem~\ref{thm:classification_general} shows that every bounded harmonic function is constant, say~\(\alpha\equiv p\), and
$$ \mu_{\alpha,\Lambda} = p \cdot \delta_{\overline{1}}+(1-p) \cdot \delta_{\overline{0}}, $$
independently of~\(\Lambda\). Thus the dependence on the broadcasting measure is nontrivial only in the finite-collision regime.

In the first part of this section, we investigate when two different moment profiles give rise to different stationary measures; in the second, we prove continuity of~\(\mu_{\alpha,\Lambda}\) with respect to the broadcast measure.

\subsection{Distinctness of the stationary measures}\label{ss_distinctness}
Recall the definition of $\Delta$ in \eqref{def:Delta}.
Since multiplying~\(\Lambda\) by a positive constant only amounts to a linear deterministic time change 
that does not affect the stationary measures, 
for simplicity we may and do henceforth assume that
\[ \Lambda[0,1]=1. \]
In this case, the moments \eqref{def:moments} are given by
\[ \lambda_\Lambda^{(k)}:=\int_{[0,1]}q^{k-1} \; \Lambda(dq), \qquad k \in \N, \]
and in particular,~\(\lambda_\Lambda^{(1)}=1\). 
Here we will investigate to what extent the~\emph{moment profile} of~$\Lambda$,
\[ \mathsf m(\Lambda) := \left( \lambda_\Lambda^{(2)},\ldots,\lambda_\Lambda^{(\Delta)} \right), \]
characterises the stationary measures~$\mu_{\alpha,\Lambda}$.
We first record the easy direction.
\begin{lemma}\label{lemm:m_imply_mu}
Let~$\alpha \in \mathscr H$ be a harmonic function, and let~$\Lambda_1,\Lambda_2$ be probability measures on~$[0,1]$. If~$\mathsf m(\Lambda_1) = \mathsf m(\Lambda_2)$ then~$\mu_{\alpha,\Lambda_1} = \mu_{\alpha,\Lambda_2}$. 
\end{lemma}
\begin{proof} 
For every~$x \in V$ and~$1 \leq k \leq \deg(x)$,
\begin{equation}\label{c_Lambda_charac}
c_\Lambda(x,k) \stackrel{\eqref{def_c_Lambda}}= \int_{[0,1]} q^{k-1}(1-q)^{\deg(x)-k} \; \Lambda(\mathrm dq) = \sum^{\deg(x)-k}_{i=0}(-1)^i \binom{\deg(x)-k}{i} \lambda^{(k+i)}_\Lambda.
\end{equation}
Since~\( \deg(x) \le\Delta\), every transition rate is determined by the moment profile. And if~$\mathsf m(\Lambda_1) = \mathsf m(\Lambda_2)$, then the generators~$\mathcal L_{\Lambda_1}$ and~$\mathcal L_{\Lambda_2}$ coincide, so the two broadcast voter models have the same semigroup~$S^{\Lambda_1}_t, S^{\Lambda_2}_t $. Since~\(\mu_{\alpha,\Lambda}\) is obtained as the weak limit of the process started from~\(\nu_\alpha\), the two limiting measures coincide.
\end{proof}

We now turn to the reverse implication and ask under which conditions~$\mu_{\alpha,\Lambda_1} = \mu_{\alpha,\Lambda_2}$ forces~$\mathsf m(\Lambda_1) = \mathsf m(\Lambda_2)$. Assume that~$\mathsf m(\Lambda_1) \neq \mathsf m(\Lambda_2)$, and let 
\begin{equation}\label{claim_k_wlg}
\kappa := \min \{ \ell \in \{2,\dots,\Delta \}: \lambda^{(\ell)}_{\Lambda_1} \neq \lambda_{\Lambda_2}^{(\ell)} \}.
\end{equation}
By Proposition~\ref{prop:main_part3}, since~$\mu_{\alpha,\Lambda_1}$ and~$\mu_{\alpha,\Lambda_2}$ satisfy~\eqref{Xi_def_t} with
$$ \mu_{\alpha,\Lambda_1}(\eta: \eta(x)=1) = \mu_{\alpha,\Lambda_2}(\eta: \eta(x)=1) = \alpha(x), \quad x \in V, $$
we have that~$\mu_{\alpha,\Lambda_1} = \mu_{\alpha,\Lambda_2}$ if and only if~$\mu_{\alpha,\Lambda_1}$ is a stationary measure for the~$\mathcal L_{\Lambda_2}$-dynamics. Since~\(\mu_{\alpha,\Lambda_1} \) is stationary for~\(\mathcal L_{\Lambda_1}\), to prove that the two measures are distinct it is therefore enough to find a local function~\(f\) such that
\begin{equation}\label{e:cond_diff_measures}
\int \mathcal L_{\Lambda_2} f \; d\mu_{\alpha,\Lambda_1}\neq0 \qquad \text{or equivalently} \qquad \int (\mathcal L_{\Lambda_1} - \mathcal L_{\Lambda_2})f\; d\mu_{\alpha,\Lambda_1} \neq0.
\end{equation}
We use the first index \(\kappa\) at which the profiles differ to construct such test functions. Recall the function~$D$ from~\eqref{Def_D_k_duality}. For a finite set~$A$ we write
$$ D_\eta(A) := \prod_{x \in A} \eta(x) \quad \text{and} \quad D_{1-\eta}(A) := \prod_{x \in A} (1-\eta(x)), \qquad \eta \in \{0,1\}^V. $$

\begin{lemma}\label{lemma_diffL12}
Let~$\mathbf x$ be a vector with distinct entries such that
\begin{equation}\label{cond_lemma_L12}
| S(\mathbf x) \cap N_v| \leq \kappa \qquad \text{for every } v \in V.
\end{equation}
Then,
$$ \left( \mathcal L_{\Lambda_1} - \mathcal L_{\Lambda_2} \right)
D_{\mathbf x}(\eta)
= \left( \lambda^{(\kappa)}_{\Lambda_1} - \lambda^{(\kappa)}_{\Lambda_2} \right) \cdot \big( I_{\mathbf x}(\eta) + O_{\mathbf x}(\eta) \big), $$
where
\begin{align*}
& I_{\mathbf x}(\eta) := \hspace{-4mm} \sum_{ \substack{x \in S(\mathbf x) \\ |S(\mathbf x) \cap N_x| = \kappa} } \hspace{-4mm} D_\eta( S(\mathbf x) \setminus N_x ) \cdot D_{1-\eta}( S(\mathbf x) \cap N_x), \\[2mm]
& O_{\mathbf x}(\eta) := \hspace{-4mm} \sum_{ \substack{x \notin S(\mathbf x) \\ |S(\mathbf x) \cap N_x| = \kappa} } \hspace{-4mm} D_\eta( S(\mathbf x) \setminus N_x ) \cdot \big[ \eta(x) D_{1-\eta}( S(\mathbf x) \cap N_x) + (-1)^\kappa(1-\eta(x)) \cdot D_\eta( S(\mathbf x) \cap N_x) \big].
\end{align*}
\end{lemma}
\begin{proof}
From the duality relationship~\eqref{duality_to_proof},
$$ \mathcal L_\Lambda D_\mathbf{x}(\eta) = \sum_{x \in V} \sum_{\varnothing \neq B \subseteq S(\mathbf x) \cap N_x } r_\Lambda( \mathbf x,x,B ) \cdot \left[ D_\eta(\mathbf x^{B \to x}) - D_\eta(\mathbf x) \right], \quad \text{for all } \eta \in \{0,1\}^V. $$
Let~$x \in V$ and let~$B \subseteq S(\mathbf x) \cap N_x$ be a nonempty set. If~$|S(\mathbf x) \cap N_x| = \kappa$, then
$$ r_{\Lambda_1}(\mathbf x,x,B) - r_{\Lambda_2}(\mathbf x,x,B) = \int_{[0,1]} q^{|B|-1}(1-q)^{\kappa-|B|} \; (\Lambda_1 - \Lambda_2)(\mathrm dq) = (-1)^{\kappa-|B|} \left( \lambda^{(\kappa)}_{\Lambda_1} - \lambda^{(\kappa)}_{\Lambda_2} \right).$$
Otherwise, if~$|S(\mathbf x) \cap N_x| < \kappa$, the corresponding difference vanishes, since only moments of order strictly smaller than~\( \kappa \) appear. We can not have~$|S(\mathbf x) \cap N_x| > \kappa$ by hypothesis. Summing over the nonempty subsets~$B \subseteq N_x \cap S(\mathbf x)$ we have
$$ (\mathcal L_{\Lambda_1} - \mathcal L_{\Lambda_2})D_\mathbf{x}(\eta) \stackrel{\eqref{gen_Lambda_CCRW}} = \sum_{ \substack{x \in V \\ |S(\mathbf x) \cap N_x| = \kappa } } \sum_{ \substack{ B \subseteq S(\mathbf x) \cap N_x \\ B \neq \varnothing } } (-1)^{\kappa-|B|} \cdot \left( \lambda^{(\kappa)}_{\Lambda_1} - \lambda^{(\kappa)}_{\Lambda_2} \right) \cdot \left( D_\eta(\mathbf x^{B \to x}) - D_\eta(\mathbf x) \right). $$
We now split according to $x$ belonging to~$S(\mathbf x)$ or not.

If~$x \in S(\mathbf x)$, for every~$B \subseteq S(\mathbf x) \cap N_x$, we have~$S(\mathbf x^{B \to x}) = S(\mathbf x) \setminus B$, 
and so
$$D_\eta(\mathbf x^{B \to x}) - D_\eta(\mathbf x) = D_\eta( S(\mathbf x) \setminus N_x ) \cdot \big( D_\eta( (S(\mathbf x) \cap N_x) \setminus B ) - D_\eta( S(\mathbf x) \cap N_x ) \big) $$
Multiplying by~$(-1)^{\kappa - |B|}$ and summing over nonempty~$B \subseteq S(\mathbf x) \cap N_x$, we 
obtain
\begin{align*}
D_\eta( S(\mathbf x) \setminus N_x ) & \cdot \left( \sum_{ \substack{ B \subseteq S(\mathbf x) \cap N_x \\ B \neq \varnothing } } (-1)^{\kappa - |B|} \cdot  D_\eta( (S(\mathbf x) \cap N_x) \setminus B ) - (-1)^{\kappa - |B|}D_\eta( S(\mathbf x) \cap N_x ) \right) \\[2mm]
& = D_\eta( S(\mathbf x) \setminus N_x ) \cdot \sum_{B \subseteq S(\mathbf x) \cap N_x } (-1)^{\kappa - |B|} \cdot  D_\eta( (S(\mathbf x) \cap N_x) \setminus B ) \\[2mm]
& = D_\eta( S(\mathbf x) \setminus N_x ) \cdot D_{1-\eta}( S(\mathbf x) \cap N_x ),
\end{align*}
where the first equality follows from the binomial identity:
$$\sum_{\varnothing \neq B \subseteq S(\mathbf x) \cap N_x} (-1)^{\kappa - |B|} = -(-1)^{\kappa}. $$
Summing over~$x \in S(\mathbf x)$ with~$|S(\mathbf x) \cap N_x| = \kappa$, we get~$I_\mathbf x(\eta)$.

If~$x \notin S(\mathbf x)$, for every~$B \subseteq S(\mathbf x) \cap N_x$, we have~$S(\mathbf x^{B \to x}) = (S(\mathbf x) \setminus B) \cup \{x\}$, and so
$$ D_\eta(\mathbf x^{B \to x}) - D_\eta(\mathbf x) = D_\eta( S(\mathbf x) \setminus N_x ) \cdot \big( \eta(x) D_\eta( (S(\mathbf x) \cap N_x) \setminus B ) - D_\eta( S(\mathbf x) \cap N_x ) \big).$$
Summing over $\emptyset \neq B \subset S(\mathbf{x})\cap N_x$ with
a similar use of the binomial identity gives~$O_\mathbf{x}(\eta)$.
\end{proof}
Notice that~$ I_{\mathbf x}\ge0 $ for every~\(\kappa\), while~$ O_{\mathbf x}\ge0 $ whenever~\(\kappa\) is even. Thus the previous lemma reduces the problem to choosing~\( \mathbf x \) so that the potentially signed term~\(O_{\mathbf x}\) is controlled and at least one of the remaining contributions has strictly positive expectation.

Before applying Lemma 4.2, we record a positivity property of the stationary measures that will be used repeatedly below.
It says that, under the finite collision property, any two distinct sites have positive probability of carrying opposite opinions whenever the harmonic profile is nontrivial. 
\begin{lemma}
Assume that~\(G\) has the finite collision property and let~$\alpha\in\mathscr H\setminus\{0,1\}. $ 
Then, for every pair
\(x,y\in V\) with $x \neq y$,
$$ \mu_{\alpha,\Lambda}(\eta: \eta(x)=1, \eta(y)=0) > 0. $$
\end{lemma}
\begin{proof}
Suppose first that~$\alpha\equiv p $ for some \(p\in(0,1)\). By duality,
$$ \mu_{p,\Lambda} \big( \eta(x)=1,\eta(y)=0 \big) = p(1-p) \cdot P_{x,y} \big( \tau_{\mathrm{coal}}^\Lambda=\infty \big). $$
The last probability is strictly positive by the finite collision property and the coupling argument of Section 3. Hence the claim follows.

Suppose now that~\(\alpha\) is nonconstant. Define, for \(x\neq y\),
$$ h(x,y) := \mu_{\alpha,\Lambda} \big( \eta(x)=1,\eta(y)=0 \big). $$
Since the one-site marginals of~\(\mu_{\alpha,\Lambda}\) are given by~\(\alpha\),
$$ h(x,y)-h(y,x) = \alpha(x)-\alpha(y). $$
Thus, since~\(\alpha\) is nonconstant, there exist~\(x_0\neq y_0\) such that~$h(x_0,y_0)>0.$ On the other hand,
$$ h(x,y) = \alpha(x) - \widehat{\mu}_{\alpha,\Lambda}(x,y). $$
By harmonicity of~\(\alpha\), stationarity of~\(\mu_{\alpha,\Lambda}\), and the generator duality,~\(h\) is harmonic for the two-particle~\(\Lambda\)-CCRW 
outside the diagonal. 
Since~\(h\ge0\), the irreducibility of the two-particle dynamics before coalescence implies that positivity at one off-diagonal state propagates to every other one. Therefore~$h(x,y)>0$ for every~$x\neq y$, which proves the claim. 
\end{proof}
The previous result extends to certain finite cylinder events that will appear below, 
as we explain next.
Let $A_1, A_0$ be two disjoint non-empty subsets of $V$ with cardinalities $|A_1|=k_1$, $|A_0| = k_0$.
Consider a system of $k=k_0+k_1$ $\Lambda$-CCRWs with initial configuration $\mathbf{x}=(x_1, \ldots, x_k) \in V^{k}$ 
such that $\{x_1, \ldots, x_{k_1}\} = A_1$ and $\{x_{k_1+1}, \ldots, x_{k}\}= A_0$.
Choose any pair \(z_1,z_0\) of distinct vertices (so that~$\mu_{\alpha,\Lambda} \big(\eta(z_1)=1,\eta(z_0)=0\big)>0$ by the previous lemma) 
and define the event 
$C = \{Y_1^1 = \cdots = Y^{k_1}_1 = z_1, Y^{k_1+1}_1 = \cdots = Y^k_1 = z_0\}$.
Then stationarity, duality, irreducibility and the Markov property give
$$ \mu_{\alpha,\Lambda} \big( \eta=1\text{ on }A_1,\, \eta=0\text{ on }A_0 \big) \ge 
P_{\mathbf{x}}(C)
\cdot \mu_{\alpha,\Lambda} \big( \eta(z_1)=1,\eta(z_0)=0 \big) >0. $$
This will imply that
every finite~\(0/1\)-configuration that occurs in the arguments below has strictly positive
probability under $\mu_{\alpha,\Lambda}$.

\begin{lemma}\label{lemma_FCP_distinct}
Let~$\alpha \in \mathscr H \setminus \{0,1\}$ be a harmonic function, and let~$\Lambda_1,\Lambda_2$ be probability measures on~$[0,1]$ with~$\mathsf m(\Lambda_1) \neq \mathsf m(\Lambda_2)$. Then~$\mu_{\alpha,\Lambda_1} \neq \mu_{\alpha,\Lambda_2}$ in each of the following cases:
\begin{enumerate}
    \item $\kappa$ is even;
    \item $\alpha \equiv p$ for some~$p \in (0,1) \setminus \{ \tfrac{1}{2}\}$;
    \item $\kappa = \Delta$;
    \item $\kappa$ is odd and there exist $v \in V$ and a vector~$\mathbf z \in V^\kappa$ with distinct entries such that
    \begin{equation}\label{property_iso}
S(\mathbf z) \subseteq N_v 
\quad \text{ and } \quad 
|N_u \cap ( S(\mathbf z) \cup \{v\} )|< \kappa \qquad \text{for every } 
       u \notin \left\{ v \right\} \cup S(\mathbf{z}).
    \end{equation}
\end{enumerate}
\end{lemma}
\begin{proof}
1. For even~$\kappa$, choose~$v$ with~$\deg(v) \ge \kappa$ and~$\mathbf x \in V^\kappa$ with distinct entries such that~$S(\mathbf x) \subseteq N_v$. Since~$|S(\mathbf x)| = \kappa$, the hypothesis~\eqref{cond_lemma_L12} of Lemma~\ref{lemma_diffL12} is automatic. Moreover~$I_\mathbf{x} = 0$, because no vertex belonging to~\(S(\mathbf x)\) can have \(\kappa\) 
neighbors in \(S(\mathbf x)\). 
Then
$$ \int O_\mathbf{x} \; d\mu_{\alpha,\Lambda_1} \ge  \mu_{\alpha,\Lambda_1} \big( \eta(v) \neq \eta(z) \; \forall z \in S(\mathbf x) \big) > 0$$
implying \eqref{e:cond_diff_measures} with $f = D_{\mathbf{x}}(\eta)$ so the conclusion follows.

2. Take~$\mathbf x$ as in the previous case, we then have~$I_\mathbf{x}=0$. 
By the previous case, we can assume that $\kappa$ is odd. 
For every vertex~\(u\notin S(\mathbf x)\) satisfying~$|S(\mathbf x)\cap N_u|=\kappa,$ we necessarily have~$S(\mathbf x) \subseteq N_u$. 
Applying Lemma~\ref{lemma_diffL12},
\begin{align*}
\int O_\mathbf{x} \; d\mu_{p,\Lambda_1} & =  \sum_{ \substack{x \notin S(\mathbf x) \\ |S(\mathbf x) \cap N_x| = \kappa} }  \left( \mu_{p,\Lambda_1}\left( \begin{array}{c}
    \eta(x)=1, \\[1mm]
    \eta(z)=0 \; \forall z\in S(\mathbf x)
\end{array}  \right) - \mu_{p,\Lambda_1}\left( \begin{array}{c}
    \eta(x)=0,  \\[1mm]
    \eta(z)=1 \; \forall z\in S(\mathbf x)
\end{array}  \right) \right). \\[2mm]
& = \sum_{ \substack{x \notin S(\mathbf x) \\ |S(\mathbf x) \cap N_x| = \kappa} } E_{(x,\mathbf x)} \left[ \mathds{1}_{E_x} \cdot \left( p(1-p)^{|\mathbf Y_\infty|} - (1-p)p^{|\mathbf Y_\infty|} \right) \right],
\end{align*}
where~\(E_x\) is the event where $Y^1_t$ never collides with $Y^2_t, \ldots, Y^\kappa_t$,
and~$|\mathbf Y_\infty| = \lim_{t \to \infty} |\{Y^2_t, \ldots, Y^\kappa_t\}|$.
The expression in braces vanishes when~\(|\mathbf Y_\infty|=1\), while for every~\(|\mathbf Y_\infty|\ge2\) it has the strict sign of~$ \frac12-p.$ By the infinite collision property,
$$  P_{(x,\mathbf x)}(E_x,\; |\mathbf Y_\infty| \ge2)>0. $$
Hence~$\int O_\mathbf x \; d\mu_{p,\Lambda_1}$ has the same strict sign,
and we conclude as before.

3. Fix one vertex~$v$ with~$\deg(v)= \Delta$ and define
$$C_v = \{ u \in V \setminus (N_v \cup \{v\}): \deg(u) = \Delta, N_u \subseteq N_v \cup \{v\} \}.$$
The set~\(C_v\) is finite, since it is contained in the ball of radius two around~$v$. 
Let~$\mathbf x$ be a enumeration of~$S = S(\mathbf x) = C_v \cup N_v \cup \{v\}$. We claim that
$$ |S\cap N_u|=\Delta \quad\Longrightarrow\quad u\in S. $$
Indeed, the left-hand side implies~$ \deg(u)=\Delta$ and~$N_u\subseteq S.$ If~\(u\notin S\), then~\(u\) cannot be adjacent to a vertex~\(w\in C_v\), since~\(N_w\subseteq N_v \cup \{v\} \) would imply~\(u\in N_v \cup \{v\} \). Hence
$$ N_u\subseteq N_v \cup \{v\}, $$
and the definition of \(C_v\) would give \(u\in C_v\), a contradiction.

Since \(\kappa=\Delta\), the hypothesis~\eqref{cond_lemma_L12} of Lemma~\ref{lemma_diffL12} holds, and the preceding observation shows that~$O_\mathbf{x} = 0$. Therefore,
\begin{align*}
\int I_\mathbf{x} \; d\mu_{\alpha,\Lambda_1} & \ge \sum_{ \substack{x \in S(\mathbf x) \\ |S(\mathbf x) \cap N_x| = \kappa} } \mu_{\alpha,\Lambda_1}( \eta(w)=1 \; \forall w \in S(\mathbf x) \setminus N_x, \; \eta(z) = 0 \; \forall z \in N_x ) \\[2mm]
& \ge \mu_{\alpha,\Lambda_1}( \eta(w)=1 \; \forall w \in C_v \cup \{v\}, \; \eta(z) = 0 \; \forall z \in N_v )>0. 
\end{align*}

4. We take~$\mathbf{x} = (v,\mathbf z)$, then~$S(\mathbf x) = S(\mathbf z) \cup \{v\}$. 
By the geometric assumption~\eqref{property_iso}, 
all vertices $u \notin S(\mathbf{x})$ satisfy $|N_u \cap S(\mathbf x)| < \kappa$.
Moreover, if $z \in S(\mathbf{z})$ then 
\[
|N_z \cap S(\mathbf{x})| = |\{v\} \cup \big(N_z \cap (S(\mathbf{z}) \setminus \{z\}) \big)| \leq \kappa. 
\]
In particular, the hypothesis~\eqref{cond_lemma_L12} holds and~$O_\mathbf{x} = 0$. 
To conclude, note that
\begin{equation*} 
\int I_\mathbf{x} \; d\mu_{\alpha,\Lambda_1} \ge  \mu_{\alpha,\Lambda_1}( \eta(v)=1, \; \eta(z) = 0 \; \forall z \in S(\mathbf z) )>0. 
\qedhere
\end{equation*}
\end{proof}

The preceding proposition does not give a complete characterization on arbitrary graphs with bounded degrees. In the remaining odd cases, we do not know whether one can always find a local test function~\(f\) such that
$$ \int (\mathcal L_{\Lambda_1} - \mathcal L_{\Lambda_2})f \; d\mu_{\alpha,\Lambda_1}\neq0. $$
More generally, we do not know whether the implication
$$ \mu_{\alpha,\Lambda_1} = \mu_{\alpha,\Lambda_2} \quad\Longrightarrow\quad \mathsf m(\Lambda_1) = \mathsf m(\Lambda_2) $$
holds on every graph with bounded degrees and the finite collision property.
Nevertheless, the geometric condition in~\eqref{property_iso} is satisfied by natural graph classes arising in applications. In particular, on trees and on the nearest neighbor Euclidean lattice~$\Z^d$, the moment profile completely characterizes the dependence of the stationary measures on the broadcasting measure. This is the content of Theorems~\ref{thm:distinctness_trees} and~\ref{thm:distinctness_Zd}, whose proofs we give next.
\begin{proof}[Proof of Theorem~\ref{thm:distinctness_trees}]
By Lemma~\ref{lemm:m_imply_mu}, equality of the moment profiles implies equality of the stationary measures. For the converse, suppose that
$\mathsf{m}(\Lambda_1)\neq\mathsf{m}(\Lambda_2)$,
and let~\(\kappa\) be the first moment at which the two profiles differ. If~\(\kappa\) is even, Lemma~\ref{lemma_FCP_distinct} applies directly. Suppose therefore that~\(\kappa\) is odd. Choose~\(v\in V\) and~$\mathbf x \in V^\kappa$ with distinct entries such that~$S(\mathbf x) \subseteq N_v$. Since there are no cycles in the graph, we have that
$$ |N_u\cap(\{v\}\cup S(\mathbf x) )| \leq 1 <\kappa \qquad\text{for every }u\neq v. $$
The geometric condition~\eqref{property_iso} of Lemma~\ref{lemma_FCP_distinct} is therefore satisfied, and~$ \mu_{\alpha,\Lambda_1} \neq \mu_{\alpha,\Lambda_2}.$ This proves the converse.    
\end{proof}

\begin{proof}[Proof of Theorem~\ref{thm:distinctness_Zd}]
Again, one implication follows from Lemma~\ref{lemm:m_imply_mu}. For the converse, let~\(\kappa\) be the first moment at which the two profiles differ. Here~$\Delta=2d.$ If~\(\kappa\) is even, Lemma~\ref{lemma_FCP_distinct} applies. Suppose that \(\kappa\) is odd. Then necessarily~$\kappa \ge 3.$ Choose~\(v\in\mathbb Z^d\) and~$\mathbf x \in V^\kappa$ with distinct entries such that~$S(\mathbf x) \subseteq N_v$. For every \(u\neq v\),
\[ |N_u\cap(\{v\}\cup S(\mathbf x))|\le 2 < \kappa. \]
Indeed, two distinct nearest neighbors of \(v\) can have a common nearest neighbor different from \(v\), but no vertex different from \(v\) can be adjacent to three distinct nearest neighbors of \(v\). Thus the geometric condition~\eqref{property_iso} in Lemma~\ref{lemma_FCP_distinct} holds, and hence~$ \mu_{p,\Lambda_1} \neq \mu_{p,\Lambda_2}.$
\end{proof}

\subsection{Continuity of the stationary measures with respect to the broadcast measure}\label{ss_continuity}
We conclude this section by studying the continuity of the stationary measures with respect to the broadcasting measure. Recall that multiplication of~\(\Lambda\) by a positive constant only amounts to a linear deterministic time change and therefore does not affect~\(\mu_{\alpha,\Lambda}\). 
Indeed, the moments \eqref{def:moments}, \eqref{def:momentvector} fully characterize $\mu_{\alpha, \Lambda}$, as discussed in Lemma~\ref{lemm:m_imply_mu}.

\begin{proof}[Proof of Theorem~\ref{thm:continuity}]
Since $\{0,1\}^V$ is compact, it suffices to prove that
\begin{equation}\label{final_limit_cylinders}
\lim_{n \to \infty} \widehat{\mu}_{\alpha,\Lambda_n}(\mathbf x) = \widehat{\mu}_{\alpha,\Lambda}(\mathbf x) \quad \text{for all } k \in \N \quad \text{and} \quad \mathbf x \in V^k.
\end{equation}
Fix~$k \in \N$ and~$\mathbf x \in V^k$. By the triangular inequality and~\eqref{ineq_mu_nu_alpha_g} from the proof of Proposition~\ref{prop:main_part3}, for every~$t \ge 0$,
\begin{align*}
|\widehat{\mu}_{\alpha,\Lambda_n}(\mathbf x) - \widehat{\mu}_{\alpha,\Lambda}(\mathbf x)| & \leq 2E_\mathbf{x}g^{{(k,\Lambda_n)}}_\mathrm{break}\left( \mathbf Y^{\Lambda_n}_t \right) + \left| E_\mathbf{x} \widehat{\nu}_\alpha \left( \mathbf Y^{\Lambda_n}_t \right) -  E_\mathbf{x} \widehat{\nu}_\alpha \left( \mathbf Y^{\Lambda}_t \right) \right| + 2E_\mathbf{x}g^{{(k,\Lambda)}}_\mathrm{break}\left( \mathbf Y^{\Lambda}_t \right),    
\end{align*}
where~$(\mathbf Y^{\Lambda_n}_t)_t$ and~$(\mathbf Y^\Lambda_t)_t$ denote the~$k$-particle~$\Lambda_n$-CCRW and the~$\Lambda$-CCRW, respectively. 

We claim that
\begin{equation}\label{eq_unif_limit_g}
\lim_{t \to \infty} \sup_{n \in \N} E_\mathbf{x} g^{(k,\Lambda_n)}_\mathrm{break}(\mathbf Y^{\Lambda_n}_t) = 0,
\end{equation}
This is a uniform version of Lemma~\ref{lemma:lim_g_Y_t}. Following that proof, the uniform limit follows from lower bounding the probability in~\eqref{away_from_0_bd} uniformly over~$n \in \N$. Recall that~$1 = \lambda^{(1)}_{\Lambda_n} \ge \lambda^{(2)}_{\Lambda_n}$ and that the total jump rate of the~$2$-particle~$\Lambda_n$-CCRW is upper-bounded by~$2\Delta$. Then
$$ \inf_{ \substack{u,v \\ d_G(u,v)\leq 2 } } \hspace{-4mm} P_{(u,v)} \left(  \begin{array}{c}
        (Y^{1,\Lambda_n}_t)_t \text{ and } (Y^{2,\Lambda_n}_t)_t \text{ coalesce before they}  \\
        \text{are for first time at graph distance~3 }
    \end{array}  \right) \ge \max \left\{ \frac{\lambda^{(2)}_{\Lambda_n}}{2\Delta},\left( \frac{1 - \lambda^{(2)}_{\Lambda_n}}{2\Delta} \right)^2  \right\}, $$
which corresponds to taking a coordinated coalescing jump or by taking two continuous steps from~$(Y^1_t)_t$ to~$Y^2_0$ and one of them being a standard coalescing jump. Noting that $\max\{a,(1-a)^2\} \ge \tfrac{3-\sqrt{5}}{2} >0$ for every~$a \in [0,1]$, we have a uniform lower bound over~$n \in \N$. As already mentioned, this is enough to obtain~\eqref{eq_unif_limit_g}.

Fix~$\epsilon > 0$. By~\eqref{eq_unif_limit_g} and Lemma~\ref{lemma:lim_g_Y_t}, there exists~$T > 0$ such that
\begin{equation}\label{final_eq_cont}
\left|\widehat{\mu}_{\alpha,\Lambda}(\mathbf x) - \widehat{\mu}_{\alpha,\Lambda_n}(\mathbf x) \right| \leq \epsilon + \left| E_\mathbf{x} \widehat{\nu}_\alpha(\mathbf Y^{\Lambda}_T) - E_\mathbf{x} \widehat{\nu}_\alpha(\mathbf Y^{\Lambda_n}_T) \right| \quad \text{for all }  n \ge 1.
\end{equation}
To bound the last term we couple~$\{ (\mathbf Y^\Lambda_t)_{t \in [0,T]},(\mathbf Y^{\Lambda_n}_t)_{t \in [0,T]}: n \in \N\}$. Let~$\mathbf x \in V^k$ and~$x \in V$, then, for each nonempty~$B \subseteq N_x \cap S(\mathbf x)$,
$$ r_{\Lambda_n}(\mathbf x,x,B) - r_\Lambda(\mathbf x,x,B) \stackrel{\eqref{eq_CCRW_rate}} = \sum^{|N_x \cap S(\mathbf x)|-|B|}_{\ell=0} (-1)^\ell \binom{|N_x \cap S(\mathbf x)|-|B|}{\ell} \cdot \left( \lambda^{(|B|+\ell)}_{\Lambda_n} - \lambda^{(|B|+\ell)}_{\Lambda} \right). $$
Hence, writing~$\delta_n := \max_{ 2 \leq \ell \leq \Delta} \left| \lambda^{(\ell)}_{\Lambda_n} - \lambda^{(\ell)}_{\Lambda} \right|$,
\begin{align*}
\sum_{x \in V} \sum_{ \substack{B \subseteq N_x \cap S(\mathbf x) \\ B \neq \varnothing} } |r_{\Lambda}(\mathbf x,x,B) - r_\Lambda(\mathbf x,x,B)| & \leq \delta_n \cdot \sum_{x \in V} \sum_{ \substack{B \subseteq N_x \cap S(\mathbf x) \\ B \neq \varnothing} } 2^{|N_x \cap S(\mathbf x)|-|B|} \\[2mm]
& =  \delta_n \cdot \hspace{-7mm} \sum_{x \in \bigcup_{ z \in S(\mathbf x)} N_z}  3^{|N_x \cap S(\mathbf x)|} - 2^{|N_x \cap S(\mathbf x)|} \leq \delta_n \cdot 3^k k\Delta.
\end{align*}
Therefore,
$$ \lim_{n \to \infty} \sup_{\mathbf x \in V^k} \sum_{x \in V} \sum_{\varnothing\neq B\subseteq N_x\cap S(\mathbf x)} \left| r_{\Lambda_n}(\mathbf x,x,B)-r_\Lambda(\mathbf x,x,B) \right| = 0. $$
We now use the standard coupling of continuous-time Markov chains. Whenever the two processes are in the same state~\(\mathbf x\), for each possible update~\((x,B)\) we run a common clock at rate
$$ r_{\Lambda_n}(\mathbf x,x,B)\wedge r_\Lambda(\mathbf x,x,B), $$
and independent clocks for the remaining parts of the two rates. Thus, as long as no unmatched clock rings, the two processes evolve together. The total rate of unmatched clocks is bounded by the preceding display, uniformly over~\(\mathbf x\). Consequently, for every fixed \(T>0\),
$$ P_{\mathbf x}\left( \mathbf Y_t^{\Lambda_n}\neq \mathbf Y^\Lambda_t \text{ for some }t\in[0,T] \right) \longrightarrow0 \qquad \text{as } n \to \infty. $$
Taking this to~\eqref{final_eq_cont} and using the fact that~$\widehat{\nu}_\alpha$ is bounded above by~1, we conclude the proof by taking~$n$ to infinity. We have proved~\eqref{final_limit_cylinders}.
\end{proof}

\subsection*{Acknowledgments}
J.A. was partially funded by FAPESP through grants 2023/13453-5 and 2025/02707-1. RSdS was partially funded by CNPq project 401314/2025-1 and FAPEMIG project APQ-00762-26.

\bibliographystyle{alpha}
\bibliography{Ref.bib}

\end{document}